\documentclass[a4paper,11pt]{amsart}
\usepackage{subfiles}

\usepackage[english,french]{babel}
\usepackage{graphicx}
\usepackage{amsmath} 
\usepackage[all]{xy}
\usepackage{amsthm}
\usepackage{amsfonts}
\usepackage{amssymb}
\usepackage{mathrsfs}
\usepackage{amsbsy}
\usepackage{mathrsfs}
\usepackage{stmaryrd}
\usepackage{amsmath}
\usepackage{mathtools}
\usepackage[utf8]{inputenc}
\usepackage{calligra}
\usepackage{mathtools}
\usepackage{aurical}
\usepackage[T1]{fontenc}

\makeatletter
\DeclareFontFamily{OMX}{MnSymbolE}{}
\DeclareSymbolFont{MnLargeSymbols}{OMX}{MnSymbolE}{m}{n}
\SetSymbolFont{MnLargeSymbols}{bold}{OMX}{MnSymbolE}{b}{n}
\DeclareFontShape{OMX}{MnSymbolE}{m}{n}{
    <-6>  MnSymbolE5
   <6-7>  MnSymbolE6
   <7-8>  MnSymbolE7
   <8-9>  MnSymbolE8
   <9-10> MnSymbolE9
  <10-12> MnSymbolE10
  <12->   MnSymbolE12
}{}
\let\llangle\@undefined
\let\rrangle\@undefined
\DeclareMathDelimiter{\llangle}{\mathopen}%
                     {MnLargeSymbols}{'164}{MnLargeSymbols}{'164}
\DeclareMathDelimiter{\rrangle}{\mathclose}%
                     {MnLargeSymbols}{'171}{MnLargeSymbols}{'171}
\makeatother

\usepackage[usenames,dvipsnames]{color}
\usepackage{color}

\RequirePackage{color}
\definecolor{bwmagenta}{rgb}{0.0,0.45,0.6}

\definecolor{bwblue}{rgb}{0.4,0.1,0.2}

\usepackage[usenames,dvipsnames]{xcolor}

\usepackage[bookmarks,colorlinks=true,pagebackref,final,breaklinks=true,pdfsubject={LaTeX}]{hyperref}
\hypersetup{
        bookmarksnumbered=true,
        linkcolor=bwblue,
        citecolor=bwmagenta,
}

\usepackage{calligra}
\makeatletter
\def\@splitop#1#2\@nil{$\mathscr{#1}\!\!$\calligra#2\,\,}
\newcommand*\DeclareCursiveOperator[2]{%
  \newcommand#1{\mathop{\mbox{\@splitop#2\@nil}}\nolimits}}
\makeatother
\DeclareCursiveOperator{\TAY}{Hess}
\DeclareCursiveOperator{\HOM}{Hom}
\DeclareCursiveOperator{\Det}{D{}et}

\newtheorem{theo}{Theorem}[section]
\newtheorem{cor}[theo]{Corollary}

\newtheorem{lemma}[theo]{Lemma}
\newtheorem{remark}[theo]{Remark}

\newtheorem{proposition}[theo]{Proposition}

\newtheorem{theorem}{Theorem}[section]

\newcommand{\quo}[1]{ \mathbf{Z}/p^{n}\mathbf{Z}  }

\newcommand{\defo}{\mathbb{T}}

\newcommand{\Sel}{Sel(\defo)}

\usepackage[OT2,T1]{fontenc}
\DeclareSymbolFont{cyrletters}{OT2}{wncyr}{m}{n}
\DeclareMathSymbol{\sha}{\mathalpha}{cyrletters}{"58}

\newcommand{\Z}{\mathbf{Z}}
\newcommand{\Q}{\mathbf{Q}}

\newcommand{\C}{\mathbf{C}}

\renewcommand{\c}{\boldsymbol{c}}

\usepackage{lipsum}

\usepackage{suffix}

\newcommand\chapterauthor[1]{\authortoc{#1}\printchapterauthor{#1}}
\WithSuffix\newcommand\chapterauthor*[1]{\printchapterauthor{#1}}

\makeatletter
\newcommand{\printchapterauthor}[1]{%
  {\parindent0pt\vspace*{-25pt}%
  \linespread{1.1}\large\scshape#1%
  \par\nobreak\vspace*{35pt}}
  \@afterheading%
}
\newcommand{\authortoc}[1]{%
  \addtocontents{toc}{\vskip-10pt}%
  \addtocontents{toc}{%
    \protect\contentsline{chapter}%
    {\hskip1.3em\mdseries\scshape\protect\scriptsize#1}{}{}}
  \addtocontents{toc}{\vskip5pt}%
}
\makeatother

\usepackage{turnstile}

\newcommand{\invlim}{\mathop{\varprojlim}\limits}
\DeclareMathOperator{\fin}{f}
\newcommand{\mint}[3]{\times \!\!\!\!\!\!\!\int^{#1}\!\!\!\!\int_{#2}^{#3} \!\! \omega_f}
\DeclareMathOperator{\rec}{rec}
\newcommand{\fpp}{{\mathfrak p}}
\DeclareMathOperator{\sign}{\mathrm{sign}}
\DeclareMathOperator{\wt}{wt}
\newcommand{\kk}{{k_{_{\!\circ}}}}
\newcommand{\smallmat}[4]{\bigl(\begin{smallmatrix}#1&#2\\#3&#4\end{smallmatrix}\bigr)}
\DeclareMathOperator{\cyc}{cyc}
\newcommand{\hW}{{\mathbb{W}}}
\newcommand{\hkappa}{\boldsymbol{\kappa}}

\newcommand{\cH}{\mathcal H}

\newcommand{\cE}{\mathcal E}

\newcommand{\cW}{\mathcal W}

\newcommand{\PP}{\mathbb{P}}
\newcommand{\Gal}{\mathrm{Gal\,}}
\newcommand{\GL}{\mathrm{GL}}

\newcommand{\Fr}{\mathrm{Fr}}

\newcommand{\ordi}{{\mathrm{ord}}}
\newfont{\gotip}{eufb10 at 12pt}

\newcommand{\cO}{{\mathcal O}}

\newcommand{\cP}{{\mathcal P}}
\newcommand{\cL}{{\mathcal L}}

\newcommand{\ra}{\rightarrow}
\newcommand{\lra}{\longrightarrow}

\newcommand{\SL}{{\mathrm {SL}}}

\newcommand{\Pic}{{\mathrm{Pic}}}

\newcommand{\testg}{\breve{g}} 
\newcommand{\testh}{\breve{h}}

\newcommand{\htf}{\breve{\mathbf{f}}}
\newcommand{\htg}{\breve{\mathbf{g}}}
\newcommand{\hth}{\breve{\mathbf{h}}}
\newcommand{\Lp}{{\mathscr{L}_p}}

\newcommand{\hf}{{\mathbf{f}}}
\newcommand{\hg}{{\mathbf{g}}}
\newcommand{\hh}{{\mathbf h}}

\newcommand{\faa}{{\mathfrak{a}}}
\newcommand{\fbb}{{\mathfrak{b}}}
\newcommand{\fc}{{\mathfrak{c}}}

\numberwithin{equation}{section}

\begin{document}

\title[Stark-Heegner points and diagonal classes]{Stark-Heegner points and diagonal classes}
\author{Henri Darmon and Victor Rotger}

\begin{abstract}
Stark-Heegner points 
are conjectural 
substitutes for Heegner points when the imaginary quadratic field of the theory of complex multiplication is replaced 
by a real quadratic field $K$. 
They are constructed  analytically as local  points  on elliptic curves with
 multiplicative reduction at a prime $p$ that remains inert in $K$,  but are conjectured 
 to be rational over ring class fields of $K$ and to satisfy a Shimura reciprocity
  law describing the action of $G_K$ on them. 
The main conjectures of \cite{darmon-hpxh} predict that any linear combination of Stark-Heegner
points weighted by the values of a ring class character $\psi$ of $K$  should belong to
 the corresponding piece of the Mordell-Weil group over the associated ring class field, 
 and should be non-trivial when $L'(E/K,\psi,1) \ne 0$. 
 Building on the
  results on families of diagonal classes described in the remaining contributions to this volume,   
   this note explains how
such linear combinations  
  arise from  global classes in the idoneous pro-$p$ 
 Selmer group, and are non-trivial when the first derivative of 
 a weight-variable $p$-adic $L$-function $\Lp(\hf/K,\psi)$  does not vanish at the  point associated to $(E/K,\psi)$.
 
 \medskip\medskip

\noindent
{\bf\em R\'esum\'e}. ---
 Les points de Stark-Heegner
g\'en\'eralisent les points de Heegner quand  un corps  quadratique r\'eel
 $K$  remplace
 le corps quadratique imaginaire de la th\'eorie de la multiplication complexe.
 La d\'emarche  $p$-adique  qui les sous-tend
 fait que   ces points  sur une courbe elliptique $E$ 
 sont \`a priori locaux, 
 d\'efinis sur   sur une 
 extension finie de $\Q_p$.
 On conjecture qu'ils   sont de nature globale, qu'ils appartiennent aux  groupes  de Mordell-Weil de
$E$ sur certains corps  de classe  de $K$, 
et qu'ils satisfont
 une loi de r\'eciprocit\'e de Shimura
 d\'ecrivant l'action de  $ G_K := {\rm Gal}(\bar K/K)$.
Les  conjectures de \cite{darmon-hpxh} pr\'edisent ainsi 
qu'une combinaison  lin\'eaire de points de Stark-Heegner
 pond\'er\'ee par les valeurs d'un caract\`ere  $\psi $ de $ G_K $  appartient  au sous-espace propre 
 correspondant du groupe Mordell-Weil de $E$
 sur le corps de classe  d\'ecoup\'e par $\psi$, 
 et qu'elle est  non triviale si et seulement si 
 $ L '(E/K, \psi, 1) \ne 0$.
 On d\'emontre que cette combinaison linéaire
  provient tout au moins 
  d'une classe globale dans la $\psi$-partie du pro-$p$ 
 groupe de Selmer de $E$, et qu'elle est non-triviale  lorsque la  d\'eriv\'ee premi\`ere d'une 
certaine fonction $L$ $p$-adique associ\'ee \`a $E$
 ne s'annule pas en $\psi$.

\end{abstract}

\address{H. D.: Montreal, Canada}
\email{darmon@math.mcgill.ca}
\address{V. R.: IMTech, UPC and Centre de Recerca Matem\`{a}tiques, C. Jordi Girona 1-3, 08034 Barcelona, Spain}
\email{victor.rotger@upc.edu}


\maketitle

\newpage
\begin{center}
{\em To Bernadette Perrin-Riou on her 65-th birthday}
\end{center}

\tableofcontents








\section{Introduction}

Let $E $ be an elliptic curve over $\Q$ of conductor $N$ and
 let $K$ be a   quadratic field of discriminant $D$ relatively prime to $N$, with    associated   
Dirichlet character $\chi_K$. 

When $\chi_K(-N)=-1$,
  the Birch and Swinnerton-Dyer conjecture    predicts
   a systematic
   supply of rational points on $E$ defined over  abelian extensions of $K$.
   More precisely, if $H$ is  any ring class field of $K$  attached to an order
   $\cO$ of $K$ 
   of conductor  prime to $D N$,  
   the Hasse-Weil $L$-function $L(E/H,s)$ 
    factors as a product
    \begin{equation}
    \label{eqn:product-L}
     L(E/H,s) = \prod_{\psi} L(E/K,\psi,s)  
     \end{equation}
     of twisted   $L$-series $L(E/K,\psi,s)$   indexed
   by the   finite order characters
$$\psi: G=\Gal(H/K)\lra L^\times,$$ 
 taking values in some fixed finite extension  $L$ of $\Q$. 
 The  $L$-series in the right-hand side of \eqref{eqn:product-L} all vanish to odd order at $s=1$, 
because they arise  from  
     self-dual Galois representations and  have
      sign $\chi_K(-N)$ in their functional 
   equations.  In particular, $L(E/K,\psi,1)=0$ for all $\psi$.
An equivariant refinement of the Birch and Swinnerton-Dyer conjecture
   predicts that the $\psi$-eigenspace $E(H)^\psi\subset E(H)\otimes L$ of the
   Mordell-Weil group for the 
   action of $\Gal(H/K)$ has dimension $\ge 1$, and hence,
    that $E(H)\otimes \Q$ contains a copy of the regular representation
    of $G$.    
    
    \medskip
   When $K$ is imaginary quadratic, this prediction  is
   largely
   accounted for by the theory of Heegner points on modular or Shimura curves, which 
   for each $\psi$ as above produces an explicit element $P_\psi \in E(H)^\psi$.
   The Gross-Zagier formula implies that $P_\psi$ is non-zero when $L'(E/K,\psi,1)\ne 0$. Thus it follows for instance that $E(H)\otimes\Q$ 
  contains a copy of the regular representation of $G$
    when 
   $L(E/H,s)$ vanishes to order $[H:K]$ at the center.

\medskip
When $K$ is real quadratic, 
the construction of non-trivial algebraic points in $E(H)$ appears to lie 
 beyond the scope of 
available techniques.  
Extending the theory of Heegner points  to this setting 
thus represents   a tantalizing challenge at
the  frontier of  our current understanding of the Birch and Swinnerton-Dyer conjecture.

\medskip
Assume from now on that $D>0$
and there is an  odd prime $p$ satisfying
\begin{equation}
\label{eqn:lrv}
 N=pM \mbox{ with }  p\nmid M,  \qquad \chi_K(p)=-1,  \qquad\chi_K(M)=1.
 \end{equation}

A conjectural construction of  Heegner-type  points, under the further 
restriction that $\chi_K(\ell) = 1$ for all 
$\ell|M$, 
was proposed in \cite{darmon-hpxh}, 
and extended to the more general  setting of 
\eqref{eqn:lrv}
 in 
  \cite{Gre}, \cite{DG}, \cite{LRV}, \cite{KPM} and \cite{Res}. It leads to a canonical collection of so-called {\em Stark-Heegner points}
$$  P_\faa \in  E(H\otimes \Q_p) = \prod_{\wp\mid p} E(H_\wp),$$ 
indexed by the  ideal classes $\faa$ of $\Pic(\cO)$, which are  regarded here as 
{\em semi-local points}, i.e.,  $[H:K]$-tuples $P_\faa = \{ P_{\faa,\wp} \}_{\wp \mid p}$ of local points in $E(K_p)$.  
This construction, and its  equivalence with the slightly different  approach of the original one, is briefly recalled in  \S \ref{sec:SH}.

As a formal consequence of the definitions (cf.\,Lemma \ref{lemma:shimura-reciprocity-formal}), the semi-local points $ P_\faa$ satisfy the Shimura reciprocity law
 $$ P_\faa^\sigma = P_{{\rm rec}(\sigma) \cdot \faa} \quad \mbox{\, for all \,} \sigma \in G,$$
where $G$ acts on the group $E(H\otimes \Q_p)$ in the natural way and  ${\rm rec}: G \ra \Pic(\cO)$ is the Artin map of global class field theory. 
 
The construction of the semi-local point $P_\faa \in \prod_{\wp\mid p} E(H_\wp)$ is purely $p$-adic analytic, relying on   a theory
  of $p$-adic integration of $2$-forms on  the product 
  $\mathcal H \times \mathcal H_p$, where $\mathcal H$ denotes Poincar\'e's complex upper half plane and $\mathcal H_p$ stands for Drinfeld's rigid analytic $p$-adic avatar of $\mathcal H$,
  the integration being performed,  metaphorically speaking, on two-dimensional regions in $\cH_p\times \cH$ 
  bounded 
  by   Shintani-type
    cycles associated to ideal classes in $K$.
      The following   statement of the 
      Stark-Heegner conjectures of loc.cit.\,is
       equivalent to \cite[Conj.\,5.6, 5.9 and 5.15]{darmon-hpxh}, and the main conjectures in  \cite{Gre}, \cite{DG}, \cite{LRV}, \cite{KPM} and \cite{Res} in the general setting of \eqref{eqn:lrv}: 
      
 \vspace{0.1in} 
 \noindent
 {\bf Stark-Heegner Conjecture}. 
 {\em 
 The semi-local points $P_\faa$ belong to the natural image of $E(H)$ in $E(H\otimes\Q_p)$, and   the $\psi$-component
 $$ P_\psi := \sum_{\faa\in\Pic(\cO)} \!\!\!\psi^{-1}(\faa) P_\faa \ \in  \ E(H\otimes\Q_p)^\psi$$
 is non-trivial if and only if $L'(E/K,\psi,1)\ne 0$.
}

\vspace{0.1in}
 
  The   Stark-Heegner Conjecture  
   has been   proved  in many cases where  $\psi$  is a {\em quadratic} ring class character.
When $\psi^2=1$,   the induced representation
 $$ V_\psi := {\rm Ind}_K^\Q \psi = \chi_1 \oplus\chi_2$$
 decomposes as the sum of two one-dimensional
 Galois representations attached to quadratic 
 Dirichlet characters satisfying
 $$ \chi_1(p) = -\chi_2(p), \qquad \chi_1(M) = \chi_2(M),$$
 and the pair $(\chi_1,\chi_2)$ can be uniquely ordered in such a way that  the $L$-series $L(E,\chi_1,s)$ and
 $L(E,\chi_2,s)$ have sign $1$ and $-1$ respectively in their functional equations.

Define the local sign $\alpha := a_p(E) $, which is equal to either $1$ or $-1$
 according to whether $E$ has split or non-split multiplicative reduction at $p$.   Let $\fpp$ 
 be a prime of $H$ above $p$,
 and let $\sigma_\fpp\in \Gal(H/\Q)$ denote the associated Frobenius element.  Because $p$ is inert in $K/\Q$, the unique prime of $K$ above 
 $p$ splits completely in $H/K$ and $\sigma_\fpp$ belongs to a
 conjugacy class of  reflections in the
 generalised dihedral group $\Gal(H/\Q)$.  It depends in an essential
 way on the choice of $\fpp$, but, 
   because
 $\psi$ cuts out an abelian extension of $\Q$, 
the Stark-Heegner point 
\begin{equation}
\label{eqn:PPsialpha}
P_\psi^\alpha := P_\psi + \alpha \cdot \sigma_\fpp  P_{\psi}
\end{equation}
does not depend on this choice. It can in fact be shown that
$$ P_\psi^\alpha = \left\{
\begin{array}{cl} 
2 P_\psi & \mbox{ if } \chi_2(p) = \alpha; \\
0 & \mbox{ if } \chi_2(p) = -\alpha.
\end{array} \right.
$$
The recent work \cite{Mok2} of Mok and  \cite{LMY} of Longo, Martin and Yan, building on the methods introduced in   \cite[Thm.~1]{BD-Ann},   \cite{Mok},   and 
\cite{longo-vigni}, 
asserts:

\vspace{0.1in}
\noindent
{\bf Stark-Heegner theorem for quadratic characters}. 
{\em Let  $\psi$  be a {\em quadratic} ring class character of conductor prime to $2D N$. Then the 
Stark-Heegner point 
$P_\psi^\alpha$ belongs to $E(H)\otimes\Q $ and is non-trivial if and only if
\begin{equation}
\label{BDgenus-thm}
L(E,\chi_1,1) \ne 0, \quad  L'(E,\chi_2,1) \ne 0, \quad \mbox{ and } \quad \chi_2(p) = \alpha.
\end{equation}   }

The principle behind the proof  of this result is to  compare $P_\psi^\alpha$ to suitable
Heegner points arising from Shimura curve parametrisations, exploiting the fortuitous circumstance 
that the field over which $P_\psi$ is conjecturally defined 
is a biquadratic extension of $\Q$ and  is thus  also   contained in  
ring class fields of imaginary quadratic fields (in many different ways).

\medskip
  The present work is concerned with the  less well understood
{\em generic}  case where
   $\psi^2 \ne 1$, when the
 induced  representation $V_\psi$ is   irreducible. Note that  $\psi$ is either totally even or totally odd,
 i.e., complex conjugation acts as a scalar $\epsilon_\psi\in \{ 1,-1\}$ on the induced representation $V_\psi$. 
   
The field which $\psi$ cuts out  cannot be embedded in any compositum of ring class fields of
imaginary quadratic fields, and the Stark-Heegner
Conjecture  
   therefore  seems
 impervious to the 
   theory of Heegner points  in this case.    
   
     
The semi-local point $P_\psi^\alpha$ 
  of \eqref{eqn:PPsialpha} now
  depends crucially on the choice of $\fpp$, 
  but it is not hard to check that its image
under the localisation homomorphism
    $$j_\fpp: E(H  \otimes \Q_p)
  \lra  E(H_\fpp) = E(K_p)  $$
  at $\fpp$ 
   is independent of this choice, up to scaling by $L^\times$ (cf.\,Lemma \ref{change-of-p}).
It is the
 local point 
  $$ P_{\psi,\fpp}^\alpha := j_\fpp(P_\psi^\alpha) \in E(H_\fpp)\otimes L = E(K_p)\otimes L$$
 which will play a key role in 
 Theorems A and B below.

  Theorems A and B are conditional
 on   
 either one  of 
 the    two    
 non-vanishing hypotheses below, which apply to
  a pair $(E,K)$ and  a choice
 of archimedean sign $\epsilon \in \{-1,1\}$.
 The first hypothesis is the counterpart, in analytic rank one, of the  non-vanishing for  simultaneous twists of  modular $L$-series 
arising as the special case of \cite[Def.~6.8]{DR2} discussed in (168) of loc.cit., where it  plays
   a similar role in the proof of the Birch and Swinnerton--Dyer conjecture for $L(E/K,\psi,s)$ when 
$L(E/K,\psi,1)\ne 0$. The main difference is that  we are now 
concerned with quadratic ring class characters for which $L(E/K,\psi,s)$ vanishes to odd rather than  to even order
at the center.

\vspace{0.3cm}
\noindent
{\bf  Analytic non-vanishing hypothesis:}  {\em Given $(E,K)$
as above, and a choice of a sign  $\epsilon \in \{1,-1\}$, 
there exists  a  quadratic
Dirichlet character $\chi$   of conductor prime to $D N$ satisfying 
$$
\chi(-1) = -\epsilon, \quad \chi \chi_K(p) = \alpha, \quad L(E,\chi,1) \ne 0, \quad  L'(E,\chi \chi_K,1) \ne 0.
$$
}

The second non-vanishing hypothesis applies to an arbitrary
 ring class character $\xi$ of $K$. 

\vspace{0.3cm}
\noindent
{\bf  Weak non-vanishing hypothesis for Stark-Heegner points:}  
{\em Given $(E,K)$ as above,  and a sign $\epsilon \in \{1,-1\}$, 
there  exists  a  ring class 
 character $\xi$   of  $K$ of conductor prime to $D N$   with
  $\epsilon_\xi = -\epsilon$    for which
  $P_{\xi,\fpp}^\alpha\ne 0$. 
}

\vspace{0.3cm}
That the  former hypothesis implies the latter  follows
 by applying the 
 Stark-Heegner theorem for  quadratic characters to the quadratic ring class character $\xi$  of $K$ attached to the pair
 $(\chi_1,\chi_2) := (\chi,\chi\chi_K)$  supplied by the analytic non-vanishing hypothesis.  
 The stronger non-vanishing hypothesis
 is singled out because it has
 the virtue of
 tying in with mainstream questions in analytic number theory on which there has been recent progress
 \cite{munshi}.
 On the other hand, the 
weak non-vanishing hypothesis is known to be   true in the classical setting of   Heegner points, when $K$ 
is imaginary quadratic. In fact, for a
given $E$ and $K$, {\em all but finitely many} of the Heegner points $P_\faa$ (as $\faa$ ranges over all ideal classes of all possible 
orders in $K$) are
of infinite order, and $P_\xi$ and $P_\xi^\alpha$
are  therefore non-trivial for infinitely many  ring class characters $\xi$, and for at least one character of any given
conductor, with finitely many exceptions.
It seems reasonable to expect that Stark-Heegner points should exhibit a similar behaviour, and the experimental evidence bears this out as one can readily verify on a software package like Pari or Magma.
In practice,  efficient algorithms for calculating Stark-Heegner points make it
easy to produce a  non-zero $P_{\xi,\fpp}^\alpha$ for any given $(E,K)$, and indeed, the extensive experiments carried out so far
  have failed  to produce even  a single example of a vanishing 
$P_\xi^\alpha$ when $\xi$ has order $\ge 3$.
Thus, while these non-vanishing hypotheses are probably difficult to prove in general,   they are expected to hold systematically. Moreover, they can easily be checked in practice for any specific 
triple $(E,K,\epsilon)$ and
therefore play  a somewhat ancillary role in studying the infinite collection of Stark-Heegner points
attached to a fixed $E$ and $K$.


Let $V_p(E) := \left(\invlim E[p^n]\right)\otimes\Q_p$ denote the Galois representation attached to $E$ and let
$$\mathrm{Sel}_p(E/H) := H_{\fin}^1(H,V_p(E))$$ 
be the pro-$p$  Selmer group of $E$ over $H$. 
The $\psi$-component of this 
 Selmer group is an $L_p$-vector space, where $L_p$ is a field containing both $\Q_p$ and $L$,
by setting 
$$ \mathrm{Sel}_p(E/H)^{\psi} := \{ \kappa \in H^1_{\fin}(H,V_p(E))\otimes_{\Q_p} L_p 
\mbox{ s.t. } 
\, \sigma \kappa = \psi(\sigma) \cdot \kappa \mbox{ for all }%
\sigma \in \Gal(H/K) \}.$$
Since $E$ is defined over $\Q$, the group $$\mathrm{Sel}_p(E/H) \simeq \oplus_\varrho H_{\fin}^1(\Q,V_p(E)\otimes \varrho)$$ admits a natural decomposition indexed by the set of irreducible representations $\varrho$  of $\Gal(H/\Q)$.
In this note we focus on the isotypic component singled out by $\psi$, namely
\begin{equation}\label{Selpsi}
\mathrm{Sel}_p(E,\psi) := 
H_{\fin}^1(\Q,V_p(E)\otimes V_\psi) 
 = \mathrm{Sel}_p(E/H)^\psi \oplus \mathrm{Sel}_p(E/H)^{\bar\psi}
\end{equation}
where Shapiro's lemma combined with the inflation-restriction sequence 
 gives the above canonical identifications.

It will be convenient to assume from now on  that $E[p]$ is irreducible as a $G_{\Q}$-module.
This hypothesis could be relaxed at the cost of  some simplicity and transparency in some
of the arguments.

\vspace{0.2cm}
\noindent
{\bf Theorem A.}
{\em   
Assume that the (analytic or weak)  non-vanishing hypothesis holds for $(E,K,\epsilon)$. Let $\psi$ be any non-quadratic ring class character of $K$ of conductor prime to $D N$,
for which $\epsilon_\psi =  \epsilon$. Then there
    is a global Selmer  class  
  $$\kappa_\psi\in \mathrm{Sel}_p(E,\psi)$$ whose  natural image in the  group $E(H_\fpp)\otimes L_p$
   of local points      agrees with   $P_{\psi,\fpp}^\alpha$. 
 }
 \vspace{0.2cm}
 
 The Selmer class mentioned in the statement above is constructed as a $p$-adic limit of diagonal classes.
 In particular, it follows from Theorem A that
\begin{equation}
\label{eqn:implication-weaker}
P_{\psi,\fpp}^\alpha \ne 0 \quad  \Rightarrow \quad \dim_{L_p} \Sel_p(E/H)^\psi \ge 1.
\end{equation}   
As a corollary, we obtain 
  a criterion for the infinitude of  $\mathrm{Sel}_p(E/H)^\psi$ in terms
  of the $p$-adic $L$-function  $\Lp(\hf/K,\psi)$ 
   constructed in \cite[\S 3]{BD-Ann}, interpolating the square roots of the
central critical values $L(f_k/K, \psi, k/2)$, as $f_k$ ranges over the weight $k\geq 2$ classical specializations of the Hida family passing through the weight two eigenform $f$ associated to $E$. The interpolation property implies that $\Lp(\hf/K,\psi)$
 vanishes at $k=2$, and its first derivative  $\Lp'(\hf/K,\psi)(2) $
  is a natural $p$-adic analogue of the derivative at $s=1$ of the classical complex $L$-function $L(f/K,\psi,s)$. The following result can thus be viewed as a  $p$-adic variant  of the Birch and Swinnerton-Dyer  Conjecture in this setting.

\vspace{0.3cm}
\noindent
{\bf Theorem B.}
{\em If $\Lp'(\hf/K,\psi)(2) \ne 0$, then $\dim_{L_p} \Sel_p(E/H)^\psi \ge 1.$
} 

\vspace{0.3cm}

Theorem B is a direct corollary of  \eqref{eqn:implication-weaker}
  in light of  the main result of \cite{BD-Ann}, recalled in Theorem \ref{GZ-SH} below,  which asserts that  $P_{\psi,\fpp}^\alpha$  is non-trivial when  $\Lp'(\hf/K,\psi)(2)\ne 0$.

\vspace{0.3cm}
 \noindent
{\bf Remark 1}. 
Assume the $p$-primary part of (the $\psi$-isotypic component of) the Tate-Shafarevich group of $E/H$ is finite. Then Theorem A shows that $P_{\psi,\fpp}^\alpha$ arises from a {\em global} point in $E(H)\otimes L_p$,  as predicted by the Stark-Heegner conjecture. Moreover, Theorem B implies that $\mbox{dim}_L\,E(H)^\psi \geq 1$ if $\Lp'(\hf/K,\psi)(2) \ne 0$.

\vspace{0.3cm}
 \noindent
{\bf Remark 2}. 
The irreducibility of $V_\psi$ when $\psi$ is non-quadratic   
shows that 
  $P_\psi^\alpha$  is non-trivial if and only if the same is true for
  $P_\psi$. The Stark-Heegner Conjecture combined with the  injectivity of the map from $E(H)\otimes L$ to 
  $E(H_\fpp)\otimes L$  suggests that $P_{\psi,\fpp}^\alpha$
  never  vanishes when $P_\psi \ne 0$, but the scenario 
  where $P_\psi^\alpha$ is  a non-trivial element of the
  kernel of $j_\fpp$ seems hard to rule out unconditionally, without assuming the Stark-Heegner conjecture
  a priori.
  

\vspace{0.3cm}
 \noindent
{\bf Remark 3}. Section \ref{sec:SH} is devoted to review the theory of Stark-Heegner points.
For notational simplicity,  \S \ref{sec:SH} has been written under the stronger Heegner hypothesis
 $$ \chi_K(p) = -1, \qquad \chi_K(\ell) = 1 \mbox{ for all  } \ell|M$$
 of \cite{darmon-hpxh}.
This section  merely  collects together  the basic notations and principal results of \cite{darmon-hpxh}, \cite{BD-Ann}, \cite{Mok2} and  \cite{LMY}. 
Exact references for the analogous results needed to cover the more general setting of \eqref{eqn:lrv} are given along the way. 
The remaining sections \S \ref{3.2}, \ref{sec:GZSH}, \ref{sec:proof-main}, \ref{3.3} and \ref{sec:AB}, which form the main body of the article, adapt without change to
 proving Theorems A and B  under the general assumption \eqref{eqn:lrv}. In particular, 
  while {\em quaternionic} modular forms need to be invoked in the general construction of Stark-Heegner points  of \cite{Gre}, \cite{DG} and \cite{LRV}, the arguments in loc.\,cit.\,only employ {\em classical elliptic modular forms} in order to deal with the general setting. 

\vspace{0.3cm}
 \noindent
{\bf Remark 4}. The proof of Theorems A and B summarized in  this note   invokes several crucial results on families of diagonal classes that are proved in the remaining contributions to this volume.
 In particular the articles \cite{BSV1} and \cite{BSV2} supply  essential ingredients in the extension of the Perrin-Riou style reciprocity laws in settings where the idoneous $p$-adic $L$-function admits
an ``exceptional zero". In a previous version of this article it was wrongly claimed that one of the key inputs, namely formula \eqref{C2} in the text, follows from one of the main results in Venerucci's   paper \cite{Ve}; the authors are  grateful to Bertolini, Seveso and Venerucci for pointing out this error and supplying  a proof of this important formula in their contributions to this volume.

\medskip
\noindent
{\bf History and connection with related work}. The first two articles  in this volume are the culmination of a project which originated in the summer of 2010 during a
two month visit by the first author to Barcelona, where, building on the  approach of \cite{BDP},
 the authors began collaborating on what eventually led to the  $p$-adic Gross-Zagier
 formula of \cite{DR1}   relating  $p$-adic
Abel-Jacobi images of diagonal cycles on a triple 
product of modular curves to the special values of  certain Garrett-Rankin triple product $p$-adic $L$-functions. 
In October of that year, they 
realized that Kato's powerful idea of varying  Galois cohomology
classes in (cyclotomic) $p$-adic families could be adapted to deforming 
the \'etale  Abel Jacobi images  of diagonal cycles, or the \'etale regulators of 
Belinson-Flach elements,  along Hida families.   
The  resulting 
 {\em generalised Kato classes}  obtained by specialising
these families
 to weight one
 seemed to  promise 
  significant  arithmetic applications,  notably for  the Birch and Swinnerton-Dyer conjecture over
ring class fields of real quadratic fields -- a setting that held a special appeal because of its connection with the
still poorly understood theory of Stark-Heegner points. 
This led the authors to formulate a program, whose broad outline was already in place by the end of
2010,  and whose key steps involved
\begin{itemize}
\item  In the setting of ``analytic rank zero",
a proof of the ``weak Birch and Swinnerton Dyer conjecture" for elliptic curves over 
$\Q$ twisted by certain Artin representations $\varrho$ of dimension $\le 4$ arising in the tensor product 
of a pair of odd two-dimensional Artin representations, i.e., the statement that
$$ L(E,\varrho,1) \ne 0 \quad \Rightarrow  (E(H)\otimes \varrho)^{G_\Q} = 0.$$
This was carried out in \cite{DR2}  and \cite{BDR}  by showing that the  generalised Kato classes 
fail to be crystalline precisely when $L(E,\varrho,1) \ne 0$. 
\item In the setting of ``analytic rank one", when $L(E,\varrho,1)=0$ it becomes natural to  compare the relevant
 generalised Kato class to   algebraic points in the 
$\varrho$-isotypic part of $E(H)$, along the lines of conjectures first formulated by Rubin (for CM elliptic curves) and
by Perrin-Riou (in the setting of Kato's work). Several precise conjectures were formulated along those
lines, notably in \cite{DLR}, guided by extensive numerical experiments conducted with Alan Lauder.
 In general, the  independent existence of such global points is tied with deep and yet unproved instances of the Birch and 
Swinnerton-Dyer conjecture, but when $\varrho$ is induced from a ring class character of a real quadratic field  $K$ and $p$ is  a prime of 
{\em multiplicative reduction} for $E$ which is inert in $K$, it becomes natural to compare the resulting generalised Kato class (a global
invariant in the Selmer group, albeit with $p$-adic coefficients) to Stark-Heegner points (which are defined purely $p$-adic analytically,
but are conjecturally motivic, with $\Q$-coefficients). 
\end{itemize}
Starting roughly in 2012,   the  idea 
of exploiting $p$-adic families of diagonal cycles and Beilinson-Flach 
elements was taken up  by several others,  motivated by a broader range of applications. 
While the authors were   fleshing out their strategy for writing the two papers appearing in this volume, they  thus
 benefitted from several  key
advances made  during this time, 
which have 
simplified and facilitated the work that is described herein,  and which it is a pleasure to acknowledge, 
most importantly:
\begin{itemize}
\item The construction of three variable cohomology classes was further developed and perfected, in the setting
of Beilinson-Flach elements by  Lei, Loeffler and Zerbes \cite{LLZ}  and several 
significant improvements were subsequently proposed, notably in the article \cite{KLZ} in which  Kings' $\Lambda$-adic sheaves play
an essential role. These provide what are  often more 
efficient and general approaches to constructing $p$-adic families of cohomology classes.

 

\item The article \cite{BSV1} by Bertolini, Seveso and Venerucci  that appears in this volume
constructs a three-variable $\Lambda$-adic class of diagonal cohomology classes by a different 
method, building on the work of Andreatta-Iovita-Stevens, and makes a more systematic study of such classes in settings
 where there is an exceptional zero, surveying a wider range of scenarios.
Although there is  some overlap between the two works as far as the general strategy is concerned,  
both present a different take on these results.  Indeed, the approach in this note eschews the methods of Andreatta-Iovita-Stevens  in favour of an approach based on the study of a collection of cycles on the cube of the modular curve $X(N)$ of full level structure. These  
cycles are of  interest in their own right, and shed a useful complementary perspective on the construction of the $\Lambda$-adic cohomology classes for the triple product.
Indeed, their study forms the basis for the 
ongoing PhD thesis of David Lilienfeldt
\cite{Li21}, and has let to  interesting open questions (cf.\,e.g.\, those that are explored in \cite{CaHs-Rank2}).

\item  Families of cohomology classes based on compatible  collections of  Heegner points 
are of course a long-standing theme in the subject, and have been taken up anew, for instance  in the more recent   works of Castella-Hsieh \cite{CaHs-Heegner}, Kobayashi \cite{Koba} and Jetchev-Loeffler-Zerbes \cite{JLZ}.

\end{itemize}

    \medskip
 \noindent
{\bf Acknowledgements.} The first author was supported by an NSERC Discovery grant. 
  The second author also acknowledges the financial support by ICREA under the ICREA Academia programme.  This project has received funding from the European Research Council (ERC) under the European
Union's Horizon 2020 research and innovation programme (grant agreement No 682152). It is a pleasure to thank M.L. Hsieh and M. Longo
 for detailed explanations of their respective recent preprints, and 
 M.~Bertolini, M.~Seveso, and R.~Venerucci  for their complementary works \cite{BSV1}, \cite{BSV2}
  appearing in this volume.
  

\section{Stark-Heegner points}
\label{sec:SH}


This section recalls briefly the  construction of Stark-Heegner points originally proposed in \cite{darmon-hpxh} and compares it with the equivalent but slightly different presentation given in the introduction. As explained in Remark 3,\,we provide the details under the running assumptions of loc.\,cit.,\,and we refer to the references quoted in the introduction for the analogous story under the more general hypothesis \eqref{eqn:lrv}.

Let $E/\Q$ be an elliptic curve of conductor $N := pM$  
with $p\nmid M$.
Since $E$ has multiplicative reduction at   $p$, 
the group $E(\Q_{p^2})$ of local points over the quadratic unramified extension
$\Q_{p^2}$  of
$\Q_p$
is equipped with Tate's 
$p$-adic uniformisation
$$ \Phi_{\rm Tate}: \Q_{p^2}^\times/q^\Z \lra E(\Q_{p^2}).$$
Let $f $ be the weight two newform 
attached to $E$ via Wiles' modularity theorem, which satisfies the usual invariance 
properties under Hecke's congruence group $\Gamma_0(N)$, and let
$$\Gamma := \left\{ \left(\begin{array}{cc} a & b \\ c & d \end{array}\right) \in \SL_2(\Z[1/p]), \qquad c \equiv 0 \pmod{M} \right\}$$
denote the associated $p$-arithmetic group, 
which acts by M\"obius transformations both on the complex upper-half plane $\cH$ and on Drinfeld's 
$p$-adic analogue $\cH_p := \PP_1(\C_p) - \PP_1(\Q_p)$. 
The main construction of Sections 1-3 of \cite{darmon-hpxh} attaches to  $f$ 
a non-trivial  {\em indefinite multiplicative integral}
$$ \cH_p \times \PP_1(\Q) \times \PP_1(\Q) \lra \C_p^\times/q^\Z, \qquad (\tau,x,y) \mapsto  \mint{\tau}{x}{y}$$
satisfying
\begin{equation}
\label{eqn:invariance-mint}
\mint{\gamma\tau}{\gamma x}{\gamma y}    = \mint{\tau}{x}{y}, \qquad \mbox{ for all } \gamma\in \Gamma,
\end{equation}
along with the requirement that 
\begin{equation}
\label{eqn:additivity-mint}
 \mint{\tau}{x}{y} = \left(\mint{\tau}{y}{x}\right)^{-1}, \qquad
\mint{\tau}{x}{y}  \times \mint{\tau}{y}{z} = \mint{\tau}{x}{z}.
\end{equation}
This function is obtained, roughly speaking, by applying the Schneider-Teitelbaum $p$-adic Poisson transform 
to a suitable  harmonic cocycle constructed from the modular symbol attached to $f$.  
It is important to note that there are in fact {\em 
 two distinct } such modular symbols, which depend on a choice of a 
sign $w_\infty = \pm 1$ at $\infty$  and are referred to as the plus and the minus modular symbols,
and therefore two distinct multiplicative integral functions, with different transformation properties 
under matrices of determinant $-1$ in $\GL_2(\Z[1/p])$.  More precisely, the  multiplicative integral associated to  $w_\infty$
  satisfies the further invariance property
$$ \mint{-\tau}{-x}{-y} = \left(\mint{\tau}{x}{y}\right)^{w_\infty}.$$
See sections 1-3 of loc.~cit., and \S 3.3. in particular, for further details.

Let 
$K$ be a real quadratic field of discriminant $D>0$, 
whose associated Dirichlet 
character $\chi_K$ satisfies the {\em Heegner hypothesis}
$$\chi_K(p) = -1, \qquad \chi_K(\ell) = 1 \mbox{ for all } \ell |M.$$
It follows that $D$ is a  quadratic 
residue modulo $M$, and we may fix a $\delta\in (\Z/M\Z)^\times$ 
satisfying  $ \delta^2 = D \pmod{M}$. Let 
$K_p\simeq \Q_{p^2}$ denote the completion of $K$ at $p$, 
and let 
$\sqrt{D}$ denote a 
chosen square root of $D$ in $K_p$.

Fix an order  $\cO$  of $K$,  of conductor $c$ 
relatively prime to $D N$.  
 The narrow Picard group 
$G_{\cO} := \Pic(\cO)$ is in bijection with the 
set of $\SL_2(\Z)$-equivalence classes of binary quadratic forms of
discriminant $D c^2$. 
 A binary quadratic form $F = Ax^2 + B xy + Cy^2$ 
  of  this discriminant is said to be a {\em Heegner form}   relative to the pair $(M,\delta)$ 
  if 
  $M$   divides $A$ and $ B\equiv \delta c  \pmod{M}$.
Every class in $G_{\cO}$ admits a representative which is a Heegner form,
  and all such representatives are equivalent under the natural action of the group $\Gamma_0(M)$.
In particular, we can write
$$ G_{\cO} =  \Gamma_0(M) \backslash \left\{ A x^2 + Bxy + Cy^2 \quad \mbox{ with } (A,B)\equiv (0, \delta c) \pmod{M} \right\}.$$
For each class $\faa := Ax^2+Bxy+Cy^2 \in G_{\cO}$ as above, 
let 
$$ \tau_{\faa} := \frac{ -B+  c \sqrt{D}}{2A} \in K_p - \Q_p \subset \cH_p,  \qquad \gamma_{\faa} := \left(\begin{array}{cc} r-Bs & -2Cs \\ 
2As & r+Bs \end{array}\right),$$
where $(r,s)$ is a primitive solution to the Pell equation $x^2-D c^2 y^2 =1$.  The matrix $\gamma_\faa\in \Gamma$ 
 has $\tau_\faa$ as a fixed point for its action on $\cH_p$.
This fact, combined with   properties \eqref{eqn:invariance-mint} and \eqref{eqn:additivity-mint}, 
implies that the period 
$$ J_\faa := \mint{\tau_\faa}{x}{\gamma_\faa x} \ \in K_p^\times/q^\Z$$ 
does not depend on the choice of $x\in \PP_1(\Q)$ that was made to define it. 
Property \eqref{eqn:invariance-mint} also shows that $J_\faa$ depends only on $\faa$ and not on the choice
of Heegner representative that was made in order to define $\tau_\faa$ and $\gamma_\faa$.
 The local point
$$ y(\faa) := \Phi_{\rm Tate}(J_\faa) \in E(K_p)$$
is called the {\em Stark-Heegner point} attached to the class $\faa\in G_\cO$.

Let $H$ denote the  narrow ring class field of $K$ attached to $\cO$, whose Galois group is canonically 
identified with $G_\cO$ via global class field theory.
Because $p$ is inert in $K/\Q$ and $\Gal(H/K)$ is a generalised dihedral group, this prime splits completely in
$H/K$. 
The set $\cP$ of primes of $H$ that lie above $p$ has cardinality $[H:K]$ and 
is endowed with a
simply transitive action of $\Gal(H/K) = G_\cO$, denoted  $(\faa,\fpp) \mapsto \faa\ast\fpp$. 

Set $K_p^{\cP} := \mathrm{Hom}(\cP,E(K_p)) \simeq K_p^{[H:K]}$.
There is a  canonical identification  
\begin{equation}
\label{eqn:semilocal}
 H \otimes \Q_p =  K_p^{\cP}, 
 \end{equation}
sending $x\in H\otimes \Q_p$ to the function $\fpp\mapsto x(\fpp):= x_\fpp$,
where $x_\fpp$ denotes  the natural image of $x$ in $H_{\fpp}=K_p$. 
The group $\Gal(H/K)$ acts compatibly 
 on both sides of \eqref{eqn:semilocal}, acting
 on the latter
via the rule
\begin{equation}
\label{eqn:galois-semilocal}
 \sigma x(\fpp)= x(\sigma^{-1}\ast \fpp).
 \end{equation}

Our fixed embedding of $H$ into $\bar\Q_p$ determines a prime
 $\fpp_0\in \cP$.
Conjecture 5.6 of \cite{darmon-hpxh} asserts that the points $y(\faa)$ are the images in $E(K_p)$ 
of global points  $ P_\faa^{?}\in E(H)$ under this embedding, and Conjecture 5.9 of loc.\,cit.\,asserts that 
 these points  satisfy the Shimura reciprocity law
 $$  P_{\fbb\faa}^{?} = {\rec}(\fbb)^{-1} P_{\faa}^{?}, \qquad \mbox{ for all } \fbb \in \Pic(\cO),$$
 where ${\rec}: \Pic(\cO) \lra \Gal(H/K)$ denotes the reciprocity map of global class field
 theory.
 
It is convenient to reformulate the   conjectures of \cite{darmon-hpxh} 
 as suggested in the introduction,
by parlaying the collection $\{y(\faa)\}$ 
 of local points in $E(K_p)$ 
into  a collection of    semi-local points 
$$P_\faa \in E(H\otimes \Q_p) =  E(K_p)^{\cP}  $$
indexed by $\faa\in G_{\cO}$. 
This is done
  by   letting $P_\faa  $ (viewed as an $E(K_p)$-valued function on 
  the set $\cP$) 
   be the  element of $E(H\otimes \Q_p)$ 
  given by
$$   (P_\faa)(\fbb \ast \fpp_0) := y(\faa\fbb),$$
so that, by definition 
\begin{equation}
\label{eqn:intermediate-action}
 P_{\fbb\faa}(\fpp) = P_{\faa}(\fbb\ast \fpp).
 \end{equation}

This point of view  has the pleasant  consequence  that the
 Shimura reciprocity  law becomes a formal consequence of the definitions: 

 \begin{lemma}
 \label{lemma:shimura-reciprocity-formal}
The  semi-local Stark-Heegner points $P_\faa \in E(H\otimes \Q_p)$ satisfy the Shimura  reciprocity law
$$ \rec(\fbb)^{-1}(P_\faa) = P_{\fbb\faa}.$$ \end{lemma}
\begin{proof} 
By  \eqref{eqn:galois-semilocal},
$$ \rec(\fbb)^{-1}(P_\faa)(\fpp) = P_\faa(\rec(\fbb)\ast \fpp) = P_{\faa}(\fbb\ast \fpp), \qquad \mbox{ for all } \fpp\in \cP.$$
But on the other hand, by \eqref{eqn:intermediate-action}
$$ P_{\faa}(\fbb \ast \fpp) =  P_{\fbb\faa}(\fpp).$$ 
The result follows from the  two displayed identities.
\end{proof}
The modular form $f$ is an eigenvector for the Atkin-Lehner involution $W_N$ acting on $X_0(N)$. 
Let $w_N$ denote its associated eigenvalue. Note that this is the negative of the sign in the functional equation for
$L(E,s)$ and hence that $E(\Q)$ is expected to have odd (resp.~ even) rank if $w_N=1$ (resp.~if $w_N=-1$). 
Recall the prime $\fpp_0$ of $H$ attached to the chosen embedding of $H$ into $\bar\Q_p$.
The frobenius element at $\fpp_0$ in $\Gal(H/\Q)$ is a reflection in this dihedral group, and is denoted by
$\sigma_{\fpp_0}$. 
\begin{proposition} 
\label{prop:forbenius-sh}
For all  $\faa \in G_{\cO}$, 
$$ \sigma_{\fpp_0} P_\faa  =  w_N   P_{\faa^{-1}}.$$
\end{proposition}
\begin{proof}
Proposition 5.10 of \cite{darmon-hpxh} asserts that 
$$ \sigma_{\fpp_0} y(\faa) = w_N y(\fc\faa)$$
for some $\fc \in G_{\cO}$. 
The definition of $\fc$ which occurs in equation (177) of loc.cit.
directly implies that
$$ \sigma_{\fpp_0} y(1) = w_N y(1), \qquad \sigma_{\fpp_0} y(\faa) = w_N y(\faa^{-1}),$$
and the result follows from this.
\end{proof}
Lemma  \ref{lemma:shimura-reciprocity-formal}
shows that 
 the 
 collection of 
Stark-Heegner points  $P_\faa$ 
is preserved under the action of  $\Gal(H/K)$, 
essentially by fiat.
A corollary of the less formal
Proposition 
\ref{prop:forbenius-sh}
is 
 the following   invariance  of the Stark-Heegner points 
 under the full action of $\Gal(H/\Q)$:
 \begin{cor}
 For all $\sigma\in \Gal(H/\Q)$ and all $\faa\in G_{\cO}$,
 $$ \sigma P_{\faa} = w_N^{\delta_\sigma} P_{\fbb}, \qquad \mbox{ for some } \fbb\in G_\cO,$$
 where 
 $$ \delta_\sigma = \left\{ \begin{array}{cl}
 0 & \mbox{ if } \sigma \in \Gal(H/K); \\
  1 & \mbox{ if } \sigma \notin \Gal(H/K). 
  \end{array}\right.
  $$
   \end{cor}
   \begin{proof} This
   follows from the fact that $\Gal(H/\Q)$ is generated by $\Gal(H/K)$ together with the 
   reflection $\sigma_{\fpp_0}$. 
   \end{proof}
 To each $\fpp\in \cP$ we have associated an embedding
 $j_\fpp: H\lra K_p$ and a frobenius element $\sigma_{\fpp}\in \Gal(H/\Q)$.
 If  $\fpp' = \sigma\ast \fpp$ is another prime in $\cP$,
 then we observe that
 \begin{equation}
 \label{eqn:dependence-on-pp}
  j_{\fpp'} = j_{\fpp} \circ \sigma^{-1}, \qquad \sigma_{\fpp'} = \sigma \sigma_{\fpp} \sigma^{-1}, \qquad j_{\fpp'} \circ \sigma_{\fpp'} = j_{\fpp} \circ \sigma_{\fpp} \circ \sigma^{-1}.
  \end{equation}
 Let $\psi:\Gal(H/K) \lra L^\times$ be a ring class character, let 
 $$ e_\psi := \frac{1}{\#G_{\cO} }\sum_{\sigma\in G_{\cO}} \psi(\sigma) \sigma^{-1} \in L[G_\cO]$$
 be the associated idempotent in the group ring, and denote by 
 $$ P_\psi := e_\psi P_1 \in  E(H\otimes\Q_p) \otimes L$$
  the  $\psi$-component of the Stark-Heegner point.
 Recall from the introduction the sign $\alpha\in \{-1,1\}$ which is equal to $1$ (resp.~$-1$) if
 $E$ has split (resp.~non-split) multiplicative reduction at the prime $p$.
 Following the notations of the introduction, write
 $$ P_\psi^\alpha =  (1+ \alpha \sigma_{\fpp}) P_\psi.$$
  
\begin{lemma}\label{change-of-p}
 The local point $j_{\fpp}(P_\psi^\alpha)$ is independent of the choice of prime $\fpp\in \cP$ that was made to define it,
 up to multiplication by a scalar in $\psi(G_{\cO}) \subset L^\times$. 
 \end{lemma}
 \begin{proof}
 Let $\fpp' = \sigma\ast \fpp$ be any other element of $\cP$.
 Then by \eqref{eqn:dependence-on-pp},
 \begin{eqnarray*}
  j_{\fpp'}(1+\alpha \sigma_{\fpp'}) P_\psi &=&  j_{\fpp} \circ \sigma^{-1} (1+ \alpha \sigma \sigma_{\fpp} \sigma^{-1}) e_\psi P_1 
 = j_{\fpp} \circ (1 + \alpha \sigma_{\fpp}) \sigma^{-1} e_\psi P_1 
 \\ &=&   \psi(\sigma)^{-1}   j_{\fpp} \circ (1 + \alpha \sigma_{\fpp}) P_\psi.
\end{eqnarray*}
 The result follows.
 \end{proof}

\medskip\noindent
{\em Examples}. 
 This paragraph describes a few numerical examples illustrating the scope and applicability of  the main results of this paper.
 By way of illustration, suppose that $E$ is an elliptic curve of prime conductor $N=p$, so that $M=1$. 
 In that special case the Atkin-Lehner sign $w_N$ is related to the local sign $\alpha$ by 
 $$ w_N = -\alpha.$$
 The following proposition reveals that the analytic non-vanishing hypothesis fails in the setting of the Stark-Heegner theorem for quadratic characters of \cite{BD-Ann} when $\epsilon = -1$:
\begin{proposition}
\label{prop:systematic-vanishing}
Let  $\psi$ be  a totally even quadratic ring class character of $K$ of  conductor prime to $N$.
Then $P_\psi^\alpha$ is trivial.
\end{proposition}
\begin{proof}
Let  $(\chi_1,\chi_2) = (\chi, \chi\chi_K)$ be the pair of even quadratic Dirichlet characters associated to $\psi$,
 ordered in such a way that 
$L(E,\chi_1,s)$ and $L(E,\chi_2,s)$ have signs $1$ and $-1$ respectively in their functional equations.
Writing $\sign(E,\chi) \in \{-1,1\}$ for the sign in
 the functional equation of the twisted 
$L$-function $L(E,\chi,s)$, it is well-known that, if the conductor of $\chi$ is relatively prime to   
$N$,  
$$ \sign(E,\chi) = \sign(E) \chi(-N) = - w_N \chi(-1) \chi(p)  = \alpha \chi(p) \chi(-1).$$
It follows that
$$ \alpha\chi_1(p) = 1, \qquad \alpha\chi_2(p)  = -1,$$
but equation 
\eqref{BDgenus-thm} in the Stark-Heegner theorem for quadratic characters implies $P_\psi^\alpha =0$.
\end{proof}
 
The systematic vanishing of $P_\psi^\alpha$ for even quadratic ring class characters of $K$ can be traced to the 
failure of the 
 analytic
non-vanishing hypothesis of the introduction, which arises for simple parity reasons.
 The failure  is expected to 
occur    essentially only when
$E$ has prime conductor $p$, i.e., when $M=1$, and never when $M$ satisfies
$\ordi_q(M)=1$ for some prime $q$. Because of Proposition \ref{prop:systematic-vanishing}, the main theorem of \cite{BD-Ann}  gives no information 
 about the 
Stark-Heegner point $P_\psi^\alpha$ attached to even quadratic ring class characters of conductor prime to 
$p$, on an elliptic curve of conductor $p$. 

On the other hand, in the setting of  Theorem A of the introduction,
 where $\psi$ has order $>2$, this phenomenon does not occur as
the non-vanishing of $P_\psi^\alpha$ and $P_\psi^{-\alpha}$ are {\em equivalent} to each other, in light of the 
irreducibility of the induced representation $V_\psi$. The numerical examples below show many instances of non-vanishing 
$P_\psi^\alpha$ for ring class characters of both even and odd parity.

\medskip
\medskip\noindent
{\em Example}.
Let $E: y^2 + y = x^3 - x $ be the elliptic curve of  
conductor $p=37$, whose Mordell-Weil group is generated by the point $(0,0)\in E(\Q)$.
Let  $K= \Q(\sqrt{5})$ be the real quadratic field 
of smallest discriminant in which $p$ is inert.
It is readily checked that $L(E/K,s)$ has a simple zero at $s=1$
and that $E(K)$ also has Mordell-Weil rank one. 
The curve $E$ has non-split multiplicative reduction at $p$ and hence $\alpha=-1$ in this case.
It is readily verified that the pair of odd characters $(\chi_1,\chi_2)$ attached to the quadratic imaginary fields
of discriminant $-4$ and $-20$ 
satisfy the three conditions in  \eqref{BDgenus-thm}, and hence the analytic non-vanishing hypothesis 
is satisfied for the triple $(E,K,\epsilon=1)$. In particular, Theorem A holds for $E$,  $K$,  and  all 
{\em even} ring class characters of $K$ of conductor prime to $37$.

Let $\cO$ be an order of $\cO_K$ with class number $3$, and let $H$ be the corresponding 
cubic extension of $K$. The prime  $\fpp$ of $H$ over $p$ and a generator $\sigma$
of $\Gal(H/K)$ can be chosen so that  the components
$$ P_1 := P_{\fpp}, \qquad P_2 := P_{\sigma\fpp}, \qquad P_3 := P_{\sigma^2\fpp} $$
in $E(H_\fpp) = E(K_p)$ of the Stark-Heegner point  in $E(H\otimes \Q_p)$ satisfy
$$ \overline{P}_1 = P_1, \qquad \overline{P}_2 = P_3, \qquad \overline{P}_3 = P_2.$$
Letting $\psi$ be the cubic character which sends $\sigma$ to $\zeta:= (1+\sqrt{-3})/2$, we find that 
\begin{eqnarray*}
 j_{\fpp}(P_\psi)  &=&  P_1 + \zeta P_2 + \zeta^2 P_3, \\
\sigma_\fpp( j_{\fpp}(P_\psi)) &=& \overline{P}_1 +  \zeta \overline{P}_2 +
\zeta^2 \overline{P}_3 = P_1 +  \zeta P_3  + 
 \zeta^2 {P}_2, \\
 j_{\fpp}(P_\psi^\alpha) &=& \sqrt{-3} \times (P_2- P_3) = \sqrt{-3} \times (P_2 -\overline{P}_2).
 \end{eqnarray*}
 The following table lists the Stark-Heegner  points $P_1$, $P_2$,  and $P_2-\overline{P}_2$
 attached to the first few orders $\cO\subset \cO_K$ of conductor $c = c(\cO)$ and 
  of class number three, calculated to a   
 $37$-adic accuracy of $2$ significant digits. (The numerical entries in the table below are
 thus  to be understood as 
 elements of $(\Z/37^2\Z)[\sqrt{5}]$.)
$$ 
\begin{array}{c|c|c|c}
c(\cO)  &  P_1   & P_2 & P_2 -\overline{P}_2  \\
\hline
 18 &(-635,  -256)  & (319 + 678\sqrt{5}, -48  1230\sqrt{5}) &  (-360, 684 + 27 \sqrt{5}) \\
 38 & (-154,447)  & (-588 + 1237 \sqrt{5} , 367  +386\sqrt{5})  & (-437, 684 + 87 \sqrt{5}) \\
  46 &  (223 , 12\cdot 37) &   (-112 +629\sqrt{5},  (-6 + 34\sqrt{5})\cdot 37)  & \infty  \\
  47 & (610,  -229) &  (539 + 71 \sqrt{5}, 10 + 439 \sqrt{5}) & (-293, 684  + 1132\sqrt{5}) \\
  54 & (533, -561)  & (679 + 984 \sqrt{5}, 
   391 +  862\sqrt{5})  &  (93, 684  + 673\sqrt{5})
  \end{array}
  $$
  Since the Mordell-Weil group of $E(K)$ has rank one, 
the  data in this table  is enough to conclude that the pro-$37$-Selmer groups of $E$ over
 the ring class fields of $K$ attached 
to the orders of conductors $18$, $38$, $47$ and $54$ have rank at least $3$. 
As for the order of conductor $46$, a calculation modulo $37^3$  reveals that
$P_2-\overline{P}_2$ is non-trivial, and hence the pro-$37$ Selmer group has rank $\ge 3$ over the ring class
field of that conductor as well.
Under   the Stark-Heegner conjecture, more is true:  the Stark-Heegner points above are
$37$-adic approximations of global points rather than mere Selmer classes.
But   recognising them as such (and thereby proving that the 
Mordell-Weil ranks are $\ge 3$) typically requires a calculations to  higher accuracy,
depending on the eventual height of the Stark-Heegner point as an algebraic point,
about which nothing is known of course a priori, and  which can behave somewhat erratically.
For example, the $x$-coordinates of the Stark-Heegner points attached to the order of conductor $47$ 
appear to satisfy the cubic polynomial 
$$x^3 - 319x^2 + 190 x + 420,$$
while those of the Stark-Heegner points for the order of conductor $46$ 
appear to satisfy the cubic polynomial 
$$ 2352347001 x^3 - 34772698791 x^2 + 138835821427 x - 136501565573 $$
with much larger coefficients, whose recognition requires a calculation to at least $7$ digits of $37$-adic accuracy.

 The table above produced many examples of non-vanishing $P_\psi^\alpha$ for $\psi$ even, and 
 in particular it verifies the non-vanishing hypothesis for Stark-Heegner points stated in the introduction,
 for the sign $\epsilon=-1$.
 This means that Theorem A is also true for {\em odd} ring class characters of $K$, even if the premise
 of \eqref{eqn:implication-weaker}
 is {\em never verified} for odd {\em quadratic} 
  characters of $K$.

\section{$p$-adic $L$-functions associated to Hida families} 
\label{3.2}

Let  $$\hf = \sum_{n\geq 1} a_n(\hf) q^n \in \Lambda_{\hf}[[q]]$$ be the Hida family of tame level $M$ and trivial tame character passing through $f$; cf.\,\cite{BD-Ann} and \cite[\S 1.3]{DR20a} for more details on the notations chosen for Hida families. 

Let $x_0 \in \cW_{\hf}^\circ$ denote the point of weight $2$ such that $\hf_{x_0}=f$. Note that $\hf_{x_0}\in S_2(N)$ is new at $p$, while for any $x\in \cW_{\hf}^\circ$ with $\wt(x)=k>2$, $\hf_x(q) = \hf_x^\circ(q)-\beta \hf_x^\circ(q^p)$ is the ordinary $p$-stabilisation of an eigenform $\hf_x^\circ$ of level $M=N/p$. We set $\hf_{x_0}^\circ = \hf_{x_0}=f$.

Let $K$ be a real quadratic field in which $p$ remains inert and all prime factors of $M$ split, and fix throughout a finite order anticyclotomic character $\psi$ of $K$ of conductor $c$ coprime to $D N$, with values in a finite extension $L_p/\Q_p$. Note that $\psi(p)=1$ as the prime ideal $p\cO_K$ is principal.

Under our running assumptions, the sign of the functional equation satisfied by the Hasse-Weil-Artin $L$-series $L(E/K,\psi,s) = L(f,\psi,s)$ is $$\varepsilon(E/K,\psi) = -1,$$ and in particular the order of vanishing of $L(E/K,\psi,s)$ at $s=1$ is odd. In contrast, at every classical point $x$ of even weight $k>2$ the sign of the functional equation satisfied by $L(\hf_x/K,\psi,s)$  is $$\varepsilon(\hf_x/K,\psi) = +1$$
and one expects generic non-vanishing of the central critical value $L(\hf_x/K,\psi,k/2)$.

In \cite[Definition 3.4]{BD-Ann},   a $p$-adic $L$-function
$$
\Lp(\hf/K,\psi) \in \Lambda_{\hf}
$$ 
associated to the Hida family $\hf$, the ring class character $\psi$ and a choice of collection of periods was defined, by interpolating the algebraic part of (the square-root of) the critical values  $L(\hf_x/K,\psi,k/2)$ for $x\in \cW_{\hf}^\circ$ with $\wt(x)=k=\kk+2\geq 2$. See also \cite[\S 4.1]{LMY} for a more general treatment, encompassing the setting considered here.

In order to describe this $p$-adic $L$-function in more detail, let $\Phi_{\hf_x,\C}$ denote the classical modular symbol associated to $\hf_x$ with values in the space $P_{\kk}(\C)$ of homogeneous polynomials of degree $\kk$ in two variables with coefficients in $\C$. The space of modular symbols is naturally endowed with an action of $\GL_2(\Q)$ and we let $\Phi_{\hf_x,\C}^+$ and $\Phi_{\hf_x,\C}^-$ denote the plus and minus eigencomponents of $\Phi_{\hf_x,\C}$ under the involution at infinity induced by $w_\infty = \smallmat 100{-1}$. 

As proved in \cite[\S 1.1]{KZ} (with slightly different normalizations as for the powers of the period $2\pi i$ that appear in the formulas, which we have taken into account accordingly), there exists a pair of collections of complex periods 
$$
\{\Omega^+_{\hf_x,\C}\}_{x\in \cW_{\hf}^\circ}, \quad \{\Omega^-_{\hf_x,\C}\}_{x\in \cW_{\hf}^\circ} \subset \C^\times
$$ 
satisfying the following two conditions:

\begin{enumerate}
\item[(i)] the modular symbols $$\Phi^+_{\hf_x} := \frac{\Phi^+_{\hf_x,\C}}{\Omega^+_{\hf_x,\C}}, \quad \Phi^-_{\hf_x} := \frac{\Phi^-_{\hf_x,\C}}{\Omega^-_{\hf_x,\C}} \quad \mbox{ take values in } \Q(\hf_x) = \Q(\{a_n(\hf_x)\}_{n\geq 1}),
$$ 
\item[(ii)] \mbox{ and } \quad
$\Omega^+_{\hf_x,\C} \cdot \Omega^-_{\hf_x,\C} = 4\pi^2 \langle \hf^\circ_x,\hf^\circ_x\rangle.$
\end{enumerate}

\vspace{0.3cm}

Note that conditions (i) and (ii) above only characterize $\Omega^{\pm}_{\hf_x,\C}$ up to multiplication by non-zero scalars in the number field $\Q(\hf_x)$. 

Fix an embedding $\bar\Q \hookrightarrow \bar\Q_p \subset \C_p$, through which we regard  $\Phi_{\hf_x}^{\pm}$ as $\C_p$-valued modular symbols.
In \cite{GS}, Greenberg and Stevens introduced measure-valued modular symbols $\mu^+_{\hf}$ and $\mu^-_{\hf}$ interpolating the classical modular symbols $\Phi_{\hf_x}^+$ and $\Phi_{\hf_x}^-$ as $x$ ranges over the classical specializations of $\hf$. 

More precisely, they show (cf.\,\cite[Theorem 5.13]{GS} and \cite[Theorem 1.5]{BD-Inv}) that for every $x\in \cW_{\hf}^\circ$, there exist $p$-adic periods 
\begin{equation}
\Omega^{+}_{\hf_x,p}, \, \Omega^{-}_{\hf_x,p}\in \C_p
\end{equation}
 such that the specialisation of $\mu^+_{\hf}$ and $\mu^-_{\hf}$ at $x$ satisfy
\begin{equation}\label{def-lambda}
 x(\mu^+_{\hf}) = \Omega^{+}_{\hf_x,p} \cdot \Phi_{\hf_x}^+,  \qquad x(\mu^-_{\hf}) = \Omega^{-}_{\hf_x,p} \cdot \Phi_{\hf_x}^-.
\end{equation}

Since no natural choice of periods $\Omega_{\hf_x,\C}^{\pm}$ presents itself, the scalars $\Omega^{+}_{\hf_x,p}$ and $\Omega^{-}_{\hf_x,p}$ are not expected to vary $p$-adically continuously. However, conditions (i) and (ii) above imply that the {\em product} $\Omega^{+}_{\hf_x,p} \cdot \Omega^{-}_{\hf_x,p}\in \C_p$ is a more canonical quantity, as it may also be characterized by the formula
\begin{equation}\label{def-lambda2}
x(\mu^+_{\hf}) \cdot x(\mu^-_{\hf}) = \Omega^{+}_{\hf_x,p}\Omega^{-}_{\hf_x,p} \cdot \frac{\Phi^+_{\hf_x,\C} \cdot \Phi^-_{\hf_x,\C}}{4\pi^2\langle \hf^\circ_x,\hf^\circ_x\rangle},
\end{equation}
which is independent of any choices of periods.

This suggests that the map $x \mapsto \Omega^{+}_{\hf_x,p}\Omega^{-}_{\hf_x,p}$ may extend to a $p$-adic analytic function, possibly after multiplying it by suitable Euler-like factors at $p$. And indeed, the following theorem is proved in one of the contributing articles of Bertolini, Seveso and Venerucci to this volume, and we refer to \cite[\S 3]{BSV2} for the proof.

\begin{theorem}\label{propo}
There exists a rigid-analytic function $\Lp(\mathrm{Sym}^2(\hf))$ on a neighborhood $U_{\hf}$ of $\cW_{\hf}$ around $x_0$ such that for all classical points $x\in U_{\hf} \cap \cW_{\hf}^\circ$ of weight $k\geq 2$:
\begin{equation}\label{BD-lambda}
\Lp(\mathrm{Sym}^2(\hf))(x) = \cE_0(\hf_x) \cE_1(\hf_x) \cdot \Omega^{+}_{\hf_x,p}\Omega^{-}_{\hf_x,p},
\end{equation}
where $\cE_0(\hf_x)$ and $\cE_1(\hf_x)$ are as in \cite[Theorem 1.3]{DR1}. Moreover, $\Lp(\mathrm{Sym}^2(\hf))(x_0)\in \Q^\times$.
\end{theorem}

\begin{remark}
The motivation for denoting $\Lp(\mathrm{Sym}^2(\hf))$ the $p$-adic function appearing above relies on the fact that $\Omega^{\pm}_{\hf_x,p}$ are $p$-adic analogues of the complex periods $\Omega^{\pm}_{\hf_x,\C}$. As is well-known, 
the product $\Omega^+_{\hf_x,\C} \cdot \Omega^-_{\hf_x,\C} = 4\pi^2 \langle \hf^\circ_x,\hf^\circ_x\rangle$ is essentially the near-central
 critical value of the classical $L$-function associated to the symmetric square of $\hf^\circ_x$. In addition to this, as M. L. Hsieh remarked to us,  it might not be difficult to show that $\Lp(\mathrm{Sym}^2(\hf))$ is a generator of Hida's congruence ideal in the sense of \cite[\S 1.4, p.4]{Hs}.
\end{remark}







The result characterizing the $p$-adic $L$-function $\Lp(\hf/K,\psi)$ alluded to above is  \cite[Theorem 3.5]{BD-Ann}, which we recall below. Although \cite[Theorem 3.5]{BD-Ann} is stated in loc.\,cit.\,only for  genus characters, the proof has been recently generalized to arbitrary (not necessarily quadratic)
ring class characters $\psi$ of conductor $c$ with $(c,DN)=1$ by Longo, Martin and Yan in \cite[Theorem 4.2]{LMY}, by
employing Gross-Prasad test vectors to
 extend Popa's formula \cite[Theorem 6.3.1]{Po} to this setting.

Let  $\mathfrak{f}_c\in K^\times$ denote the explicit constant introduced at the first display of \cite[\S 3.2]{LMY}. It only depends on the conductor $c$ and its square lies in $\Q^\times$. 

\begin{theorem}
\label{bd2} 
The $p$-adic $L$-function $L_p(\hf/K,\psi)$ satisfies the 
following interpolation property: for all $x\in \cW_\hf^\circ$ of weight $\wt(x)=k=\kk+2 \geq 2$, we have
$$
\Lp(\hf/K,\psi)(x) = \mathfrak{f}_{\hf, \psi}(x) \times L(\hf^\circ_x/K,\psi,k/2)^{1/2}
$$
where
$$
\mathfrak{f}_{\hf,\psi}(x) = (1-\alpha_{\hf_x}^{-2} p^{\kk})\cdot  \frac{ \mathfrak{f}_c \cdot (Dc^2)^{\frac{\kk+1}{4}}  (\frac{\kk}{2})!}{(2 \pi i)^{\kk/2}} \cdot \frac{\Omega^{\epsilon_\psi}_{\hf_x,p}}{\Omega^{\epsilon_\psi}_{\hf_x,\C}}.
$$
\end{theorem}

\section{A $p$-adic Gross-Zagier formula for Stark-Heegner points}\label{sec:GZSH}

One of the main theorems of  \cite{BD-Ann} is a formula for the derivative of $\Lp(\hf/K,\psi)$ at the point $x_0$, relating it to the formal group logarithm of a Stark-Heegner point. This formula shall be crucial for relating these points to generalized Kato classes and eventually proving our main results.

\begin{theorem}\label{GZ-SH}
The $p$-adic $L$-function  $\Lp(\hf/K,\psi)$ vanishes at the point $x_0$ of weight $2$ and
\begin{equation}\label{thm-bd2}
\frac{d}{dx} \Lp(\hf/K,\psi)_{|x=x_0} = \frac{1}{2}\log_p(P_\psi^\alpha).
\end{equation}

\end{theorem}

\begin{proof}
The vanishing of $\Lp(\hf/K,\psi)$ at $x=x_0$ is a direct consequence of the assumptions and definitions, because $x=x_0$ lies in the region of interpolation of the $p$-adic $L$-function and therefore $\Lp(\hf/K,\psi)(x_0)$ is a non-zero multiple of the central critical value $L(f/K,\psi,1)$. This $L$-value vanishes as remarked in the paragraph right after \eqref{eqn:product-L}.
 
The formula for the derivative follows verbatim as in the proof of \cite[Theorem 4.1]{BD-Ann}. See also \cite[Theorem 5.1]{LMY} for the statement in the generality required here. Finally, we refer to \cite{longo-vigni} for a formulation and proof of this formula in the setting of quaternionic Stark-Heegner points, under the general assumption of \eqref{eqn:lrv}.
\end{proof}

\section{Setting the stage}
\label{sec:proof-main}

In this section we set the stage for the proofs of Theorems A and B by introducing a particular choice of triplet of eigenforms $(f,g,h)$ of weights $(2,1,1)$.
Let $E/\Q$ be an elliptic curve having multiplicative reduction at a prime $p$ and set $\alpha = a_p(E) = \pm 1$. Let
$$
\psi: \Gal(H/K) \lra L^\times
$$ be an anticyclotomic character of a real quadratic field $K$ satisfying the hypotheses stated in the introduction. 

In particular we assume that a prime ideal $\fpp$ above $p$ in $H$ has been fixed and either of the {\em non-vanishing hypotheses} stated in loc.\,cit.\,holds; these hypotheses give rise to a character $\xi$ of $K$ having parity opposite to that of $\psi$ that we fix for the remainder of this note, satisfying that the local Stark-Heegner point $P_{\xi,\fpp}^\alpha$ is non-zero. 

As shown in \cite[Lemma 6.9]{DR2}, there exists a (not necessarily anti-cyclotomic) character  $\psi_0$ of finite order of $K$ and conductor prime to $D N_E$ such that 
\begin{equation}\label{psi1}
\psi_0/\psi'_0 = \xi/\psi.
\end{equation}
Since by hypothesis $\xi/\psi $ is totally odd, it follows that $\psi_0$ has mixed signature $(+,-)$ with respect to the two real embeddings of $K$.

Let $\mathfrak c\subset \mathcal O_K$ denote the conductor of $\psi_0$
and let $\chi$
denote the odd central Dirichlet character of $\psi_0$. Let $\chi_K$ also denote the quadratic Dirichlet  character associated to $K/\Q$.

Let $f\in S_2(pM_f)$ denote the modular form associated to $E$ by modularity. Likewise, set
$$M_g=D c^2\cdot \mathrm{N}_{K/\Q}(\mathfrak c) \quad \mbox{and} \quad M_h=D\cdot \mathrm{N}_{K/\Q}(\mathfrak c)
$$ 
and define the eigenforms 
$$
g=\theta(\psi_0  \psi)\in S_1(M_g,\chi \chi_K) \quad \mbox{ and } \quad h=\theta(\psi_0^{-1})\in S_1(M_h,\chi^{-1} \chi_K)$$ to be the theta series associated to the characters $\psi_0 \psi$ and $\psi_0^{-1}$, respectively. 

Recall from the introduction that $E[p]$ is assumed to be irreducible as a $G_{\Q}$-module. This implies that the mod $p$ residual Galois representation attached to $f$ is irreducible, and thus also non-Eisenstein mod $p$. The same claim holds for $g$ and $h$ because $\psi$ and $\xi$ have opposite signs and $p$ is odd, hence $\xi \not\equiv \psi^{\pm  1}  \,(\mathrm{mod}\, p)$.

Note that $p\nmid M_f M_g M_h$. 
As in previous sections, we let $M$ denote the least common multiple of $M_f$, $M_g$ and $M_h$.  The Artin representations $V_g$ and $V_h$ associated to $g$ and $h$ are both odd and unramified at the prime $p$. Since $p$ remains inert in $K$, the arithmetic frobenius $\Fr_p$ acts on $V_g$ and $V_h$ with eigenvalues
$$
\{ \alpha_g,  \beta_g \} = \{\zeta, \, -\zeta\}, \qquad  \{\alpha_h, \beta_h \} = \{ \zeta^{-1}, -\zeta^{-1}\},
$$
where $\zeta$ is a root of unity satisfying $\chi(p)=-\zeta^2$. 

In light of \eqref{psi1} we have $\psi_0  \psi /\psi_0 = \psi$ and $\psi_0  \psi /\psi_0' = \xi$, hence the tensor product of $V_g$ and $V_h$ decomposes as
\begin{equation}\label{psi-psi1}
V_{gh} = V_g \otimes V_h \simeq \mathrm{Ind}^{\Q}_K(\psi) \oplus \mathrm{Ind}^{\Q}_K(\xi) \quad \mbox{as } G_{\Q}\mbox{-modules}
\end{equation}
and
$$
V_g = V_g^{\alpha_g} \oplus V_g^{\beta_g}, \quad V_h = V_h^{\alpha_h} \oplus V_h^{\beta_h}, \quad V_{gh} = \bigoplus_{(a,b)}  V_{gh}^{a b} \quad \mbox{as } G_{\Q_p}\mbox{-modules}
$$
where $ (a,b)$ ranges through the four pairs $(\alpha_g,\alpha_h),(\alpha_g,\beta_h),(\beta_g,\alpha_h),(\beta_g,\beta_h)$. Here $V_g^{\alpha_g}$, say, is the $G_{\Q_p}$-submodule of $V_g$ on which $\Fr_p$ acts with eigenvalue $\alpha_g$, and similarly for the remaining terms.



Let $W_p$ be an arbitrary self-dual Artin representation with coefficients in $L_p$ and factoring through the Galois group of a finite extension $H$ of $\Q$. Assume $W_p$ is unramified at $p$.
There is a canonical isomorphism
\begin{eqnarray}
\label{eqn:coho-induced}
 H^1(\Q, V_p(E)\otimes W_p) &\simeq&  (H^1(H, V_p(E))\otimes W_p)^{\Gal(H/\Q)} \\
 \nonumber
  &=&  \mathrm{Hom}_{\Gal(H/\Q)}(W_{p}, H^1(H,V_p(E))),
  \end{eqnarray}
  where the the second equality follows from the self-duality of $W_p$. Kummer theory gives rise to a homomorphism
  \begin{equation}
  \label{eqn:def-delta-global}
\delta: E(H)^{W_p} := \mathrm{Hom}_{\Gal(H/\Q)}(W_p, E(H)\otimes L_p)\lra   H^1(\Q,V_p(E) \otimes W_p). 
\end{equation}
    For each rational prime $\ell$, the 
    maps \eqref{eqn:coho-induced} and \eqref{eqn:def-delta-global}
    admit local counterparts
 \begin{eqnarray*}
    H^1(\Q_{\ell}, V_p(E)\otimes W_p) &\simeq&    \mathrm{Hom}_{\Gal(H/\Q)}(W_{p}, \oplus_{\lambda|\ell} H^1(H_\lambda,V_p(E))), \\
    \delta_\ell: \left(\oplus_{\lambda|\ell} E(H_\lambda)\right)^{W_p}  &\lra&
    H^1(\Q_{\ell},V_p(E) \otimes W_p),
    \end{eqnarray*}
   for which the following diagram commutes:
\begin{equation}
   \label{eqn:comm-diag-resp}
   \xymatrix{  
    E(H)^{W_p} \ar[r]^{\delta \qquad} \ar[d]^{\mathrm{res}_\ell}  & H^1(\Q, V_p(E)\otimes W_p)  \ar[d]^{\mathrm{res}_\ell} \\
    \left(\oplus_{\lambda|\ell} E(H_\lambda)\right)^{W_p} \ar[r]^{\delta_\ell \quad} &
    H^1(\Q_\ell, V_p(E)\otimes W_p).  }
\end{equation}

For primes $\ell\ne p$, it follows from \cite[(2.5) and (3.2)]{Ne2} that $H^1(\Q_\ell,V_p(E)\otimes W_p)=0$. (We warn however that if we were working with integral coefficients, these cohomology groups may contain non-trivial torsion.) For $\ell=p$, the Bloch-Kato submodule $H^1_{\fin}(\Q_p,V_p(E)\otimes W_p)$ is the subgroup of $H^1(\Q_p, V_p(E)\otimes W_p)$ formed by classes of {\em crystalline} extensions of Galois representations of $V_p(E)\otimes W_p$ by $\Q_p$. It may also be identified with the image of the local connecting homomorphism $\delta_p$.


\begin{lemma}\label{partial} 
There is a natural isomorphism of $L_p$-vector spaces
\begin{equation*}
H^1_{\fin}(\Q_p,V_p(E)\otimes W_p) =  H^1_{\fin}(\Q_p,V^+_f\otimes W_p^{\Fr_p=\alpha}) \oplus H^1(\Q_p,V^+_f\otimes W_p/W_p^{\Fr_p = \alpha}),
\end{equation*}
where recall $\alpha=a_p(E)= \pm 1$.
\end{lemma}

 \begin{proof} We firstly observe that $H^1_{\fin}(\Q_p,V_p(E)\otimes W_p) = H^1_{\mathrm{g}}(\Q_p,V_p(E)\otimes W_p)$ by e.g.\,\cite[Prop. 2.0 and Ex. 2.20]{Be}, because $V_p(E)\otimes W_p$ contains no unramified {\em submodule}.
  As shown in \cite[Lemma , p.125]{Fl}, it follows that
 $$
 H^1_{\fin}(\Q_p,V_p(E)\otimes W_p) = \mathrm{Ker}\big( H^1(\Q_p,V_p(E)\otimes W_p) \lra H^1(I_p,V^-_p(E)\otimes W_p)\big)
 $$
 is the kernel of the composition of the homomorphism in cohomology induced by the natural projection $V_p(E) \lra V_p^-(E)$ and restriction to the inertia subgroup $I_p \subset G_{\Q_p}$. 
 
The long exact sequence in Galois cohomology arising from the exact sequence 
$$
0 \ra V^+_p(E) \lra V_p(E) \lra V^-_p(E) \ra 0
$$
 shows that the kernel of the map $H^1(\Q_p,V_p(E)\otimes W_p) \lra H^1(\Q_p,V^-_p(E)\otimes W_p)$ is naturally identified with $H^1(\Q_p,V^+_p(E)\otimes W_p)$. We have $H^1(I_p,\Q_p(\psi \varepsilon_{\cyc})) = 0$ for any nontrivial unramified character $\psi$. Besides, it follows from e.g.\,\cite[Example 1.4]{DR20a} that $H^1_{\fin}(\Q_p,\Q_p(\varepsilon_{\cyc}))  = \ker\big( H^1(\Q_p,\Q_p(\varepsilon_{\cyc})) \ra H^1(I_p,\Q_p(\varepsilon_{\cyc}))\big)$ is a line in the two-dimensional space $H^1(\Q_p,\Q_p(\varepsilon_{\cyc}))$, which Kummer theory identifies with $\Z_p^\times  \hat\otimes_{\Z_p} \Q_p$ sitting inside $\Q_p^\times \hat\otimes_{\Z_p} \Q_p$.

Note that $V^+_p(E)= L_p(\psi_f \varepsilon_{\cyc})$ and $V_p^-(E) \simeq L_p(\psi_f)$ where $\psi_f$ is the unramified quadratic character of $G_{\Q_p}$ sending $\Fr_p$ to $\alpha$. The lemma follows.
\end{proof}


 The Selmer group $\Sel_p(E, W_p)$ is defined as 
\begin{equation*}\label{defSelmer}
\Sel_p(E,W_p)  := \{ \lambda \in  H^1(\Q,V_p(E)\otimes W_p): \mathrm{res}_p(\lambda)\in H^1_{\fin}(\Q_p,V_p(E)\otimes W_p)\}.
\end{equation*}

Here $\mathrm{res}_p$ stands for the natural  map in cohomology induced by restriction from $G_\Q$ to $G_{\Q_p}$.

\section{Factorisation of $p$-adic $L$-series}
\label{3.3}

The goal of this section is proving a factorisation formula of
 $p$-adic $L$-functions which shall be crucial in the proof
  of our main theorems. 

Keep the notations introduced in the previous section and recall in particular the sign $\alpha :=a_p(f)\in \{ \pm 1\}$ associated to $E$. Let $g_\zeta$ and $h_{\alpha \zeta^{-1}}$ denote the ordinary $p$-stabilizations of $g$ and $h$ on which the Hecke operator $U_p$ acts with eigenvalue 
\begin{equation}\label{alphabeta}
\alpha_g :=\zeta \quad \mbox{and} \quad \alpha_h :=\alpha\zeta^{-1},
\end{equation}
respectively. 

Let $\hf$, $\hg$  and $\hh$ be the Hida families of tame levels $M_f$, $M_g$, $M_h$ and tame characters $1$, $\chi \chi_K$, $\chi^{-1} \chi_K$ passing respectively through $f$, $g_\zeta$ and $h_{\alpha \zeta^{-1}}$. The existence of these families is a theorem of Wiles \cite{W}, and their uniqueness follows from a recent result of Bella\"iche and Dimitrov \cite{BeDi} (note that the main theorem of loc.\,cit.\,indeed applies because $\alpha_g\ne \beta_g$, $\alpha_h\ne \beta_h$ and $p$ does not split in $K$). Let $x_0$, $y_0$, $z_0$ denote the classical points in $\cW_{\hf}$, $\cW_{\hg}$ and $\cW_{\hh}$ respectively such that $\hf_{x_0}=f$, $\hg_{y_0}=g_\zeta$ and $\hh_{z_0}=h_{\alpha \zeta^{-1}}$.

As explained in \cite{DR1}, \cite{DR2} and recalled briefly in \cite[(5.1)]{DR20a} in this volume, 
associated to a choice
$$\htf \in S^{\ordi}_{\Lambda_{\hf}}(M)[\hf], \quad \htg\in S^{\ordi}_{\Lambda_{\hg}}(M,\chi \chi_K)[\hg], \quad \hth\in S^{\ordi}_{\Lambda_{\hh}}(M,\chi^{-1}\chi_K)[\hh]$$ 
of $\Lambda$-adic test vectors of tame level $M$ there is a three-variable $p$-adic $L$-function $\Lp^f(\htf, \htg, \hth)$. Among such choices, Hsieh \cite{Hs} pinned down a particular choice of test vectors with optimal interpolation properties (cf.\,loc.\,cit.\,and \cite[Prop.\,5.1]{DR20a} for more details), which we fix throughout this section.

Define 
\begin{equation}\label{Lpf}
\Lp^f(\htf,\testg_\zeta, \testh_{\alpha \zeta^{-1}}) \in \Lambda_{\hf}
\end{equation}
to be the one-variable $p$-adic $L$-function arising as the restriction of $\Lp^f(\htf, \htg, \hth)$ to the rigid analytic curve $\cW_\hf \times \{y_0,z_0\}$.

In addition, recall the $p$-adic $L$-functions described in 
\S \ref{3.2} associated to the twist of $E/K$ by an anticyclotomic character of $K$, and set
 $\mathfrak{f}_{\cO}(\kk) :=  (Dc^2)^{\frac{1-k}{2}}/\mathfrak{f}_c^2$, where $\mathfrak{f}_c$ is the 
 constant introduced at the first display of \cite[\S 3.2]{LMY}. Note that the rule $k \mapsto \mathfrak{f}_{\cO}(\kk) $ extends to an Iwasawa
  function, that we continue to denote $\mathfrak{f}_{\cO}$, because $p$ does not divide $Dc^2$. Recall also
   the rigid-analytic function $\Lp(\mathrm{Sym}^2(\hf))$ in a neighborhood $U_{\hf} \subset \cW_{\hf}$ of $x_0$
    introduced in \eqref{BD-lambda}.

\begin{theorem}
\label{factorisation-DR} 
The following factorization of $p$-adic $L$-functions holds in $\Lambda_{\hf}$:
$$
 \Lp(\mathrm{Sym}^2(\hf)) \times \Lp^f(\htf,\testg_\zeta, \testh_{\alpha \zeta^{-1}}) = \mathfrak{f}_{\cO} \cdot \Lp(\hf/K,\psi) \times \Lp(\hf/K,\xi).
$$

\end{theorem}

\begin{proof} 

It follows from \cite[Prop.\,5.1]{DR20a} that  $\Lp^f(\htf,\testg_\zeta, \testh_{\alpha \zeta^{-1}})$ satisfies the following interpolation property for all $x\in \cW_{\hf}^\circ$ of weight $k\geq 2$:
\begin{equation*}
\Lp^f(\htf,\testg_\zeta, \testh_{\alpha \zeta^{-1}})(x)  =   (2 \pi i)^{-k}  \cdot   (\frac{\kk}{2}!)^2 \cdot  \frac{1-\alpha^{-2}_{\hf_x} p^{\kk}}{1-  \beta_{\hf_x}^2  p^{1-k}} \cdot \frac{L(\hf^\circ_x,g,h,\frac{k}{2})^{1/2}}{\langle \hf^\circ_x,\hf^\circ_x\rangle}.
\end{equation*}

Besides, it follows from Theorem \ref{bd2} that the product of $\Lp(\hf/K,\psi)$ and $\Lp(\hf/K,\xi)$ satisfies that for all $x\in \cW_{\hf}^\circ$ of weight $k\geq 2$:
$$
\Lp(\hf/K,\psi) \Lp(\hf/K,\xi)(x) = \mathfrak{f}_{\hf,\psi}(x) \cdot\mathfrak{f}_{\hf,\xi}(x) \times L(\hf^\circ_x/K,\psi,k/2)^{1/2} \cdot L(\hf^\circ_x/K,\xi,k/2)^{1/2}
$$
where  
$$
\mathfrak{f}_{\hf,\psi}(x) \cdot\mathfrak{f}_\xi (x) =  (1-\alpha_{\hf_x}^{-2} p^{\kk})^2 \cdot \frac{\mathfrak{f}^2_c \cdot (Dc^2)^{\frac{\kk+1}{2}} \cdot (\frac{\kk}{2})!^2}{(2 \pi i)^{\kk}} \cdot \frac{\Omega^+_{\hf_x,p}\Omega^-_{\hf_x,p}}{\Omega^{+}_{\hf_x,\C} \Omega^{-}_{\hf_x,\C}}.
$$

A direct inspection to the Euler factors shows that 
for all $x\in \cW_{\hf}^\circ$ of weight $k\geq 2$:
\begin{equation}
L(\hf^\circ_x,g,h,k/2) = L(\hf^\circ_x/K,\psi,k/2)\cdot L(\hf^\circ_x/K,\xi,k/2).
\end{equation}

Recall finally from Theorem \ref{propo} that the value of $\Lp(\mathrm{Sym}^2(\hf))$ at a point $x\in U_{\hf} \cap \cW_{\hf}^\circ$ is $$\Lp(\mathrm{Sym}^2(\hf))(x) = (1-  \beta_{\hf_x}^2  p^{1-k})(1-\alpha_{\hf_x}^{-2} p^{\kk}) \Omega^+_{\hf_x,p}\Omega^-_{\hf_x,p}.$$

Combining the above formulae together with the equality
$$
\Omega^+_{\hf_x,\C} \cdot \Omega^-_{\hf_x,\C} = 4\pi^2 \langle \hf^\circ_x, \hf^\circ_x\rangle,
$$
described in \S \ref{3.2}, it follows that the following formula holds for all $x\in \cW_{\hf}^\circ$ of weight $k\geq 2$:
\begin{equation*}\label{fac}
 \Lp(\mathrm{Sym}^2(\hf))(x) \times \Lp^f(\htf,\testg_\zeta, \testh_{\alpha \zeta^{-1}})(x)  = \lambda_{\cO}(\kk) \cdot \Lp(\hf/K,\psi)(x) \times \Lp(\hf/K,\xi) (x).
\end{equation*}

Since $\cW_{\hf}^\circ$ is dense in $\cW_{\hf}$ for the rigid-analytic topology, the factorization formula claimed in the theorem follows. 
\end{proof}

Recall from Theorem \ref{bd2}  that $\Lp(\hf/K,\psi)$ and $\Lp(\hf/K,\xi)$ both vanish at $x_0$ and
\begin{equation}\label{thm-bd2bis}
\frac{d}{dx} \Lp(\hf/K,\psi)_{|x=x_0} = \frac{1}{2}\cdot \log_p(P_\psi^\alpha), \quad \frac{d}{dx} \Lp(\hf/K,\xi)_{|x=x_0} = \frac{1}{2}\cdot  \log_p(P_\xi^\alpha).
\end{equation}

By Theorem \ref{propo}, $\Lp(\mathrm{Sym}^2(\hf))(x_0)\in \Q^\times$. It thus follows from Theorem \ref{factorisation-DR} that the order of vanishing of $\Lp^f(\htf^\vee, \testg_{\zeta}, \testh_{\alpha \zeta^{-1}})$ at $x=x_0$ is at least two and
\begin{equation}\label{--}
\frac{d^2}{dx^2} \Lp^f(\htf^\vee,\testg_\zeta, \testh_{\alpha \zeta^{-1}})_{|x=x_0} = C_1\cdot \log_p(P_\psi^\alpha)\cdot  \log_p(P_\xi^\alpha),
\end{equation}
where $C_1$ is a non-zero simple algebraic constant.

As recalled at the beginning of this article, $P_{\xi,\fpp}^\alpha$ is non-zero. We can also suppose that $P_{\psi,\fpp}^\alpha$ is non-zero, as otherwise there is nothing to prove. Hence \eqref{--} shows that the order of vanishing of $\Lp^f(\htf^\vee, \testg_{\zeta}, \testh_{\alpha \zeta^{-1}})$ at $x=x_0$ is exactly two.

\section{Main results}\label{sec:AB}

Let us now explain the proofs of the main theorems stated in the introduction by invoking the results proved in previous sections in combination with some of the main statements proved in the remaining contributions to this volume.

Let
$$
\kappa(\hf, \hg, \hh) \in H^1(\Q,\mathbb{V}^\dag_{\hf \hg \hh}(M))
$$
be the $\Lambda$-adic global cohomology class introduced in \cite[Def.\,5.2]{DR20a}. 

Define $\mathbb{V}^\dag_{\hf g h}(M)$ as the $\Lambda_{\hf}[G_{\Q}]$-module obtained by specialising the $\Lambda_{\hf \hg \hh}[G_{\Q}]$-module $\mathbb{V}^\dag_{\hf \hg \hh}(M)$ at $(y_0,z_0)$. Let
\begin{equation}\label{211}
\kappa(\hf, g_\zeta, h_{\alpha \zeta^{-1}}) := \nu_{y_0,z_0}\kappa(\hf, \hg, \hh) \in H^1(\Q,\mathbb{V}^\dag_{\hf g h}(M))
\end{equation}
denote the specialisation of $\kappa(\hf, \hg, \hh)$ at $(y_0,z_0)$, and 
$$
\kappa(f, g_\zeta, h_{\alpha \zeta^{-1}})  \in H^1(\Q,V_{f g h}(M)) 
$$
denote the class obtained by specializing \eqref{211} further at $x_0$.

Let us analyze the above class locally. According to the discussion preceding Lemma \ref{partial}, it follows that $\mathrm{res}_\ell \,(\kappa(f, g_\zeta, h_{\alpha \zeta^{-1}})) = 0$ at every prime $\ell \ne p$. 

In order to study it at $p$, write $\kappa_p(f, g_\zeta,h_{\alpha \zeta^{-1}}) := \mathrm{res}_p\,\kappa (f, g_\zeta,h_{\alpha \zeta^{-1}})\in H^1(\Q_p,V_f\otimes V_{gh}(M))$.

After setting $V_{gh}^{a b} = V_g^a\otimes V_h^b$, we find that there is a natural decomposition
\begin{equation}\label{ab}
H^1(\Q_p,V_p(E)\otimes V_{gh}) = \bigoplus_{(a,b)} H^1(\Q_p,V_p(E)\otimes V_{gh}^{a b})
\end{equation}
where $ (a,b)$ ranges through the four pairs $(\alpha_g,\alpha_h),(\alpha_g,\beta_h),(\beta_g,\alpha_h),(\beta_g,\beta_h)$. Analogous decompositions hold for the various Galois cohomology groups appearing in this section. Given a class $\kappa \in H^1(\Q_p,V_p(E)\otimes V_{gh}(M))$, we shall denote $\kappa^{ab}$ for its projection to the corresponding $(a,b)$-component.

Note that
 \begin{equation}\label{products}
 \alpha_g \alpha_h = \beta_g \beta_h = \alpha,  \qquad \alpha_g \beta_h = \beta_g \alpha_h = -\alpha.
 \end{equation}
 Hence, according to Lemma \ref{partial}, $\kappa(f, g_\zeta, h_{\alpha \zeta^{-1}})$ lies in the Bloch-Kato finite submodule of $H^1(\Q,V_{f g h}(M))$ if and only if
 \begin{itemize}
\item[(i)]   $\kappa_p(f,g_\zeta,h_{\alpha \zeta^{-1}})^{\alpha_g \beta_h}$ and  $\kappa_p(f,g_\zeta,h_{\alpha \zeta^{-1}})^{\beta_g \alpha_h}$ lie in $H^1(\Q_p,V^+_p(E)\otimes V_{gh}(M))$, 

\item[(ii)] $\kappa_p(f,g_\zeta,h_{\alpha \zeta^{-1}})^{\alpha_g \alpha_h}$ and  $\kappa_p(f,g_\zeta,h_{\alpha \zeta^{-1}})^{\beta_g \beta_h}$ lie in $H^1_{\fin}(\Q_p,V^+_p(E)\otimes V_{gh}(M))$.
 
 \end{itemize}

By \cite[Proposition 1.5.8]{DR20a}, the local class $\kappa_p(f, g_\zeta, h_{\alpha \zeta^{-1}}) $ is the specialization at $(x_0,y_0,z_0)$ of a $\Lambda$-adic cohomology class with values in the $\Lambda$-adic representation $\mathbb{V}_{\hf \hg \hh}^+(M)$, which recall is defined as the span in $\mathbb{V}_{\hf \hg \hh}^\dag(M)$ of (suitably twisted) triple tensor products of the form $\mathbb{V}^{\pm}_{\hf} \otimes\mathbb{V}^\pm_{\hg}\otimes \mathbb{V}^\pm_{\hh}$, with at least two $+$'s in the exponents.

Since $V_g^{\beta_g} = V_g^+$ and $V_g^{\alpha_g} = V_g^-$, and similarly for $V_h$, it follows from the very definition of $\mathbb{V}_{\hf \hg \hh}^+(M)$ that the $(\alpha_g,\alpha_h)$-component of $\kappa_p(f,g_\zeta,h_{\alpha \zeta^{-1}})$ in $H^1(\Q_p,V_f\otimes V^{\alpha_g \alpha_h}_{gh}(M))$ vanishes --this yields a fortiori claim (ii) for the $(\alpha_g,\alpha_h)$-component. The same reasoning also yields that the $(\alpha_g,\beta_h)$ and $(\beta_g,\alpha_h)$-components of the projection of $\kappa_p(f,g_\zeta, h_{\alpha \zeta^{-1}})$ to $H^1(\Q_p,V^-_f\otimes V_{gh}(M))$ vanish, and hence (i) holds.

It only remains to analyze the $(\beta_g,\beta_h)$-component $\kappa_p(f, g_\zeta,h_{\alpha \zeta^{-1}})$. For this purpose we define the $\Lambda_{\hf}[G_{\Q_p}]$-modules
$$
\mathbb{W} :=  \mathbb{V}_{\hf, \beta \beta}(M) := \mathbb{V}_{\hf}(M)( \underline{\varepsilon}_{\hf}^{-1/2})\otimes V_{gh}^{\beta_g \beta_h}(M), \quad 
$$
$$
\mathbb{W}^- :=   \mathbb{V}^-_{\hf, \beta \beta}(M) := \mathbb{V}^-_{\hf}(M)( \underline{\varepsilon}_{\hf}^{-1/2})\otimes V_{gh}^{\beta_g \beta_h}(M).
$$

It follows from \eqref{alphabeta} that $V_{gh}^{\beta \beta} = L_p(\alpha)$ is the one-dimensional representation afforded by the character of $\Gal(K_p/\Q_p)$ sending $\Fr_p$ to $\alpha=a_p(E)$. Hence $\hW^-$ is the sub-quotient of $\mathbb{V}^\dag_{\hf g h}(M)$ that is isomorphic to several copies of $\Lambda_{\hf}(\Psi_{\hf}^{g h}  \underline{\varepsilon}_{\hf}^{-1/2})$, where  as in \cite[(1.5.5)]{DR20a},
$\Psi_{\hf}^{g h}$ denotes the unramified character of $G_{\Q_p}$  satisfying 
$$
\Psi_{\hf}^{g h}(\Fr_p) = {\bf a}_p(\hf) a^{-1}_p(\hg_1) a^{-1}_p(\hh_1) = \alpha \cdot {\bf a}_p(\hf).
$$ 

Let
\begin{equation}\label{kap}
\hkappa_p^f(\hf,g_\zeta,h_{\alpha \zeta^{-1}})  \in  H^1(\Q_p, \hW), \quad \hkappa_p^f(\hf,g_\zeta,h_{\alpha \zeta^{-1}})^-  \in  H^1(\Q_p, \hW^-)
\end{equation}
denote the image of  $\kappa_p (\hf, g_\zeta, h_{\alpha \zeta^{-1}})$ under the map induced by the projection $\mathbb{V}^+_{\hf g h}(M) \ra \hW = \mathbb{V}_{\hf, \beta \beta}(M) $, and further to $\hW^- =  \mathbb{V}^-_{\hf, \beta \beta}(M) $ respectively. 

Equivalently and in consonance with our notations, $\hkappa_p^f(\hf,g_\zeta,h_{\alpha \zeta^{-1}})^-$ is the specialization at $(y_0,z_0)$ of the local class $\hkappa_p^f(\hf,\hg,\hh)^-$ introduced in \cite[(1.5.8)]{DR20a}
and invoked in \cite[Theorem 1.5.1]{DR20a}. Hence \cite[Theorem 1.5.1]{DR20a} applies and asserts that the following identity holds in $\Lambda_{\hf}$ for any triple $(\htf,\htg,\hth)$ of $\Lambda$-adic test vectors:

\begin{equation}\label{PR2}
\langle \cL_{\hf,\hg \hh}(\hkappa_p^f(\hf,g_\zeta,h_{\alpha \zeta^{-1}})^-), \eta_{\htf^*} \otimes \omega_{\testg^*_\zeta} \otimes \omega_{\testh^*_{\alpha\zeta^{-1}}}\rangle  =  \Lp^{f}(\htf^\vee,\testg_\zeta, \testh_{\alpha \zeta^{-1}}).
\end{equation}

Let now $\kappa_p^f(f,g_\zeta,h_{\alpha \zeta^{-1}})$  and $\kappa_p^f(f,g_\zeta,h_{\alpha \zeta^{-1}})^-$ denote the specializations at $x_0$ of the classes in \eqref{kap}. According to our previous definitions, we have 
\begin{equation}\label{kappa-bb}
\kappa_p(f,g_\zeta,h_{\alpha \zeta^{-1}})^{\beta_g \beta_h} = \kappa_p^f(f,g_\zeta,h_{\alpha \zeta^{-1}}).
\end{equation}

Since $a_p(f)=\alpha\in \{ \pm 1	\}$ and $\underline{\varepsilon}_{\hf}(x_0)=1$, it follows from the above description of $\hW$ and the character $\Psi_{\hf}^{g h}$ that
$\hW(x_0)\simeq V_p(E_+)(M)$  as $G_{\Q_p}$-modules,
where $E_+$ is the (trivial or quadratic) twist of $E$ given by $\alpha$. 
Hence 
$
\kappa_p^f(f,g_\zeta,h_{\alpha \zeta^{-1}})\in H^1(\Q_p,V_p(E_+)(M)).
$

The Bloch-Kato dual exponential and logarithm maps associated to the $p$-adic representation $V_p(E_+)(M)$ take values in a space $L_p(M)$ consisting of several copies of the base field $L_p$. Given a choice of test vectors, it gives rise to a projection $L_p(M) \lra L_p$. We shall denote  by a slight abuse of notation
$$
\log_{\mathrm{BK}}: H^1_{\fin}(\Q_p,V_p(E_+)(M)) \lra L_p
$$
the composition of  the Bloch-Kato logarithm with the projection to $L_p$.



The following fundamental input comes  from the main results due to Bertolini, Seveso and Venerucci in this volume, and we refer to \cite{BSV1} and \cite{BSV2} for the detailed proof; here we just content to point out to precise references in loc.\,cit. As explained in the introduction, in a previous version of this paper formula \eqref{C2} below was wrongly attributed to \cite{Ve}.

\begin{theorem}\label{corolo}(Bertolini, Seveso, Venerucci) The local class $ \kappa_p^f(f,g_\zeta,h_{\alpha \zeta^{-1}})$ is crystalline and
\begin{equation}\label{C2}
\frac{d^2}{dx^2}  \Lp^{f}(\htf^\vee,\testg_\zeta, \testh_{\alpha \zeta^{-1}})_{|x=x_0}  \, = \, C_2 \cdot \log_{\mathrm{BK}}(\kappa_p^f(f,g_\zeta,h_{\alpha \zeta^{-1}}))  
\end{equation}
for some nonzero rational number $C_2\in \Q^\times$.
\end{theorem}

Indeed, the first claim of the above theorem follows from \cite[Theorem B]{BSV1}: since $L(f,g,h,1)=0$ it follows from the equivalence between (a) and (c) of \cite[\S 9.4]{BSV1} that the dual exponential map vanishes on $\kappa_p^f(f,g_\zeta,h_{\alpha \zeta^{-1}})$ --note that the improved class $\kappa^*_g(f,g_\zeta,h_{\alpha \zeta^{-1}})$ of loc.\,cit.\,is simply a non-zero multiple of $\kappa(f,g_\zeta,h_{\alpha \zeta^{-1}})$  in our setting, because of \eqref{products}. 
  This amounts to saying that the class is crystalline.  Formula \eqref{C2} follows from  \cite[Proposition 2.2]{BSV2} combined with \eqref{PR2}.

In light of \eqref{kappa-bb} and the above discussion, the above theorem implies that $\kappa(f, g_\zeta, h_{\alpha \zeta^{-1}})$ belongs to the Selmer group $H^1_{\fin}(\Q,V_{f g h}(M))$, as conditions (i) and (ii) above are fulfilled. 

Recall from \eqref{psi-psi1} that $V_{gh}=V_\psi \oplus V_\xi$ decomposes as the direct sum of the induced representations of $\psi$ and $\xi$.  Write
\begin{eqnarray}\label{kpsixi}
\kappa_\psi(f,g_\zeta,h_{\alpha \zeta^{-1}}) \in H^1_{\fin}(\Q,V_p(E)\otimes V_\psi(M)), \\ \nonumber \kappa_\xi(f,g_\zeta,h_{\alpha \zeta^{-1}}) \in H^1_{\fin}(\Q,V_p(E)\otimes V_\xi(M))
\end{eqnarray}
for the projections of the class $\kappa(f, g_\zeta, h_{\alpha \zeta^{-1}})$ to the corresponding quotients. We denote as in the introduction
 $$ 
 \kappa_\psi^\alpha(f,g_\zeta,h_{\alpha \zeta^{-1}}) = 
 (1+ \alpha \sigma_\fpp) \kappa_\psi(f,g_\zeta,h_{\alpha \zeta^{-1}}) \in 
 H^1_{\fin}(H,V_p(E)(M))^{\psi \oplus \bar\psi}
$$
the component of $\kappa_\psi(f,g_\zeta,h_{\alpha \zeta^{-1}})$ on which $\sigma_{\fpp}$ acts with eigenvalue $\alpha$,
and likewise with $\psi$ replaced by the auxiliary character $\xi$. 
 
 \begin{lemma}\label{aa0} We have
 $$
 \log_{E,\fpp}  \kappa_\psi^\alpha(f,g_\zeta,h_{\alpha \zeta^{-1}}) =  \log_{E,\fpp}  \kappa_\xi^\alpha(f,g_\zeta,h_{\alpha \zeta^{-1}}).
 $$
 \end{lemma}
 
 \begin{proof} We may decompose the local class  $$\kappa_p:=\kappa_p(f,g_\zeta,h_{\alpha \zeta^{-1}}) = (\kappa_p^{\alpha_g\alpha_h},\kappa_p^{\alpha_g\beta_h},\kappa_p^{\beta_g \alpha_h},\kappa_p^{\beta_g \beta_h})$$ 
 in $H^1(\Q_p,V_f\otimes V^{\alpha_g \alpha_h}_{gh}(M))$ as the sum of four contributions with respect to the decomposition \eqref{ab} afforded by the eigen-spaces for the action of $\sigma_\fpp$.
In addition to that, $\kappa_p$ also decomposes as 
$$
 \kappa_p = (\kappa_{\psi,p},\kappa_{\xi,p}) \in H^1_{\fin}(\Q_p,V_p(E)\otimes V_\psi(M)) \oplus H^1_{\fin}(\Q_p,V_p(E)\otimes V_\xi(M)),
$$ 
where $\kappa_{\psi,p}$, $\kappa_{\xi,p}$ are the local components at $p$ of the classes in \eqref{kpsixi}. An easy exercise in linear algebra shows that
\begin{equation}\label{kaakbb}
\kappa_p^{\alpha_g\alpha_h} = \kappa^\alpha_{\psi,p} - \kappa^\alpha_{\xi,p}, \quad \kappa_p^{\beta_g\beta_h} = \kappa^\alpha_{\psi,p} + \kappa^\alpha_{\xi,p}.
\end{equation}
Since we already proved that $\kappa_p^{\alpha_g\alpha_h}=0$,  the above display implies that $\kappa^\alpha_{\psi,p} = \kappa^\alpha_{\xi,p}$ are the same element. The lemma follows. 
\end{proof}

Let
$$
\log_{\beta_g \beta_h}: H_{\fin}^1(\Q_p,V_f\otimes V_{gh}(M)) \stackrel{\mathrm{pr}_{\beta_g \beta_h}}{\lra} H_{\fin}^1(\Q_p,V_f\otimes V^{\beta_g \beta_h}_{gh}(M)) 
 \stackrel{\log_{\mathrm{BK}}}{\lra} L_p
$$
denote the composition of the natural projection to the $(\beta_g,\beta_h)$-component with the Bloch-Kato logarithm map associated to the $p$-adic representation $V_f\otimes V^{\beta_g \beta_h}_{gh}(M) \simeq V_{f_+}(M)$ and the choice of test vectors. Note that $H^1_{\fin}(\Q_p,V_p(E_+))= H^1_{\fin}(\Q_p,\Q_p(1))$, which as recalled in \cite[Example 1.1.4 (c)]{DR20a} is naturally identified with  the completion of $\Z_p^\times$, and the Bloch-Kato logarithm is nothing but the usual $p$-adic logarithm on $\Z_p^\times$ under this identification. 
Lemma \ref{aa0} together with the second identity in \eqref{kaakbb} imply that
\begin{equation*}
(i) \qquad \log_{E,\fpp}  \kappa_\psi^\alpha(f,g_\zeta,h_{\alpha \zeta^{-1}}) = \log_{\beta_g \beta_h}(\kappa_p(f,g_\zeta,h_{\alpha \zeta^{-1}})) .
\end{equation*}

Thanks to \eqref{C2} we have
$$
(ii) \qquad \log_{\beta_g \beta_h}(\kappa_p(f,g_\zeta,h_{\alpha \zeta^{-1}}))  
 = \frac{d^2}{dx^2}  \Lp^{f}(\htf^\vee,\testg_\zeta, \testh_{\alpha \zeta^{-1}})_{|x=x_0}  \pmod{L^\times}.$$

Finally, fix $(\htf,\htg,\hth)$ to be Hsieh's choice of $\Lambda$-adic test vectors satisfying the properties stated in Theorem \ref{factorisation-DR}.
Recall from \eqref{--} that, with this choice, we have
%
$$
(iii) \qquad \frac{d^2}{dx^2} \Lp^f(\htf^\vee,\testg_\zeta, \testh_{\alpha \zeta^{-1}})_{|x=x_0} =  \log_p(P_\psi^\alpha)\cdot  \log_p(P_\xi^\alpha)  \pmod{L^\times}.
$$

Define
$$ \kappa_\psi := \log_{E,\fpp}(P_\xi^\alpha)^{-1} \times \kappa_\psi^\alpha(f,g_\alpha,h_\alpha).$$
It follows from the combination of  (i)-(ii)-(iii) that $\kappa_\psi$ fulfills the claims stated in Theorem A, and hence the theorem is proved.

Theorem B also follows, because the non-vanishing of the first derivative $\frac{d}{dx}\Lp(\hf/K,\psi)_{|x=x_0}$ implies that $P_{\psi,\fpp}^\alpha \ne 0$. Theorem A then implies that the class $\kappa_\psi\in H^1_{\fin}(H,V_p(E)(M))^{\psi \oplus \bar\psi}$ is non-trivial.

\let\oldaddcontentsline\addcontentsline
\renewcommand{\addcontentsline}[3]{}
 
\let\addcontentsline\oldaddcontentsline


\begin{thebibliography}{BDLP20a}

 \bibitem[Be09]{Be} 
 Joel Bellaiche,  {\em An introduction to the conjecture of Bloch and Kato}, available at  {\tt http://www.people.brandeis.edu/~jbellaic/BKHawaii5.pdf}. 

\bibitem[BeDi16]{BeDi} 
Joel Bella\"iche and Mladen Dimitrov, {\em On the eigencurve at classical weight one points},
Duke Math. J.  {\bf 165 } (2016), no. 2, 245--266.

\bibitem[BD07]{BD-Inv}
Massimo Bertolini and Henri Darmon, {\em Hida families and rational points on elliptic curves}, 
Invent. Math. {\bf 168} (2007), no. 2, 371--431.

\bibitem[BD09]{BD-Ann}
Massimo Bertolini and Henri Darmon, {\em The rationality of Stark-Heegner points over genus fields of real quadratic fields},  Annals Math. {\bf 170} (2009) 343--369. 

\bibitem[BD14]{BD-Kato1}
Massimo Bertolini and Henri Darmon, {\em Kato's Euler system and rational points on elliptic curves I: a $p$-adic Beilinson formula} Israel J. Math. {\bf 199} (2014), 163-188.



\bibitem[BDP13]{BDP}
Massimo Bertolini, Henri Darmon, and Kartik  Prasanna, {\em  Generalised Heegner cycles and $p$-adic
Rankin $L$-series}, Duke Math J. {\bf 162}, No. 6,  (2013) 1033--1148.
 

\bibitem[BDR15]{BDR}
Massimo Bertolini, Henri Darmon, and Victor Rotger. {\em Beilinson-Flach elements and Euler systems II: the Birch-Swinnerton-Dyer conjecture for Hasse-Weil-Artin L-series}
J. Algebraic Geom. {\bf 24} (2015), no. 3, 569--604. 



\bibitem[BSVa]{BSV1}
Massimo Bertolini, Marco Seveso, and Rodolfo Venerucci,  {\em
 Reciprocity laws for balanced diagonal classes}, in this volume.
 
 \bibitem[BSVb]{BSV2}
 Massimo Bertolini, Marco Seveso, and Rodolfo  Venerucci,  {\em
Balanced diagonal classes and rational points on elliptic curves}, in this volume. 




\bibitem[CS18]{CaHs-Heegner} 
 Francesc Castella and  Ming-Lun Hsieh,  {\em Heegner cycles and p-adic L-functions}, 
Math. Annalen {\bf 370} (2018), 567--628.

\bibitem[CS20]{CaHs-Rank2} 
Francesc Castella and Ming-Lun Hsieh, {\em On the non-vanishing of generalized Kato classes for elliptic curves of rank two}, preprint available at \url{https://web.math.ucsb.edu/~castella}.

\bibitem[CW77]{coates-wiles}
John Coates and Andrew Wiles. {\em On the conjecture of Birch and Swinnerton-Dyer}.
 Invent. Math. 39 (1977), no. 3, 223--251. 



\bibitem[Da01]{darmon-hpxh}
Henri Darmon, 
{\em  Integration on $\cH_p\times \cH$ and arithmetic applications}. 
Ann. of Math. {\bf 154} (2001), 589--639.

\bibitem[DLR15]{DLR}
Henri Darmon, Alan Lauder,  and Victor Rotger,  
{\em Stark points and $p$-adic iterated integrals attached to modular forms of weight one.}
Forum of Mathematics, Pi, (2015), Vol.~3, e8, 95 pages.


\bibitem[DR14]{DR1}
Henri Darmon and Victor Rotger, 
{\em Diagonal cycles and Euler 
systems I: a $p$-adic Gross-Zagier formula}, 
 Annales Scientifiques 
de l'Ecole Normale Sup\'erieure {\bf 47}, n. 4 (2014), 779--832.


\bibitem[DR16]{DR3}
Henri Darmon and Victor Rotger,
{\em Elliptic curves of rank two and generalised Kato classes},
Research in Mathematics, Special issue in memory of Robert Coleman,  3:27 (2016).


\bibitem[DR17]{DR2}
Henri Darmonand Victor Rotger, {\em 
Diagonal cycles and Euler systems II:   the Birch and Swinnerton-Dyer conjecture
 for Hasse-Weil-Artin $L$-series},
Journal of the American Mathematical Society,  {\bf 30} Vol. 3, (2017) 601--672.

\bibitem[DRb]{DR20a}
Henri Darmon and Victor Rotger, {\em $p$-adic families of diagonal cycles}, in this volume.

 
\bibitem[DG12]{DG}
Samit Dasgupta and Matthew Greenberg, $\cL$-invariants and Shimura curves, {\em Algebra and Number Th.} {\bf 6} (2012), 455--485.

 \bibitem[Fl90]{Fl} 
Matthias  Flach, {\em A generalisation of the Cassels-Tate pairing}, 
J. reine angew. Math. {\bf 412} (1990), 113--127.

\bibitem[Gr09]{Gre} 
Matthew Greenberg, {\em Stark--Heegner points and the cohomology 
of quaternionic Shimura varieties}, Duke Math. J. {\bf 147} (2009), no. 
3, 541--575.
 
\bibitem[GS93]{GS}
Ralph Greenberg and Glenn Stevens, 
{\em $p$-adic $L$-functions and $p$-adic periods of modular forms}.
Invent. Math. {\bf 111} (1993), 407--447.
 
 

\bibitem[Hs20]{Hs}
Ming-Lun Hsieh, {\em Hida families and p-adic triple product L-functions}, American J. Math., to appear.


\bibitem[JLZ20]{JLZ} 
Dimitar Jetchev, David Loeffler and Sarah Zerbes, Heegner points in Coleman families, to appear in Proc. London Math. Soc.



\bibitem[Ka04]{Ka}
 Kazuya Kato, {\em $p$-adic Hodge 
theory and values of zeta functions of modular forms},  Ast\'erisque {\bf 295}, 2004.

\bibitem[Ki15]{Ki}
Guido Kings,
{\em Eisenstein classes, elliptic Soul\'e elements and the $\ell$-adic elliptic polylogarithm}. 
London Math. Soc. Lecture Note Ser., {\bf 418}, Cambridge Univ. Press, 2015.


\bibitem[KLZ17]{KLZ}
Guido Kings, David Loeffler, and Sarah Zerbes, 
{\em Rankin-Eisenstein classes and explicit reciprocity laws}, 
Cambridge J. Math. {\bf 5} (2017), 1--122.
 


\bibitem[Ko20]{Koba}
Shinishi Kobayashi, {\em A p-adic interpolation of generalized Heegner cycles and integral Perrin-Riou twist}, preprint available at \url{https://sites.google.com/view/shinichikobayashi}.

\bibitem[KPM18]{KPM} 
 Daniel Kohen and Ariel Pacetti, with an appendix by Marc Masdeu,
 {\em On Heegner points for primes of additive reduction ramifying in the base field}, 
 Trans. Amer. Math. Soc., Trans. Amer. Math. Soc. {\bf 370} (2018), 911--926.

\bibitem[KZ84]{KZ} 
Winfried Kohnen and Don Zagier, {\em Modular forms with rational periods}, in
"Modular forms" (Durham, 1983), 197--249, 
Ellis Horwood Ser. Math.
Appl.: Statist. Oper. Res., Horwood, Chichester, 1984.


\bibitem[Li]{Li21} David Lilienfeldt, McGill Ph.D Thesis, in progress. 



\bibitem[LLZ14]{LLZ}
Antonio Lei,  David Loeffler,  and Sarah Zerbes. {\em  Euler systems for Rankin-Selberg convolutions of modular forms}. 
Ann. of Math. (2) {\bf 180} (2014), no. 2, 653--771.  

\bibitem[LRV12]{LRV} 
Matteo Longo, Victro Rotger,  and Stefano Vigni,  
{\em On rigid analytic uniformizations of Jacobians of Shimura curves},
Amer. J. Math. {\bf 134}  (2012), no. 5, 1197--1246. 

\bibitem[LMY17]{LMY} 
Matteo Longo, Kimball Martin, and Hu Yan,  
{\em Rationality of Darmon points over genus fields of non-maximal orders}, preprint 2017. 


\bibitem[LV14]{longo-vigni}
Matteo Longo and Stefano Vigni,  {\em The rationality of quaternionic Darmon points over genus fields
 of real quadratic fields}, Int. Math. Res. Not. {\bf 13} (2014), 3632--3691. 
 
 
   
\bibitem[Mo11]{Mok}
Chung-Pang Mok,  
{\em Heegner points and $p$-adic $L$-functions for elliptic curves over certain totally real fields}, Comment. Math. Helv. {\bf 86} (2011), 867--945. 

\bibitem[Mo17]{Mok2}
Chung-Pang Mok,  
{\em On a theorem of Bertolini-Darmon about rationality of Stark-Heegner points over genus fields of a real quadratic field}, preprint 2017.

\bibitem[Mu12]{munshi}
Ritabrata Munshi,
{\em A note on simultaneous nonvanishing twists}. 
J. Number Theory {\bf 132} (2012), no. 4, 666--674. 


\bibitem[Ne98]{Ne2} 
Jan Nekov\'a\v{r}, \textit{$p$-adic Abel-Jacobi maps and $p$-adic heights}, in {\em The Arithmetic and Geometry of Algebraic Cycles} (Banff, Canada, 1998), CRM Proc. Lect. Notes {\bf 24} (2000), 367--379.

 

\bibitem[Po06]{Po}
Alex Popa, {\em 
Central values of Rankin $L$-series over real quadratic fields.}
Compos. Math. {\bf 142 }  (2006), 811--866. 

 
\bibitem[Re15]{Res}
Juan Restrepo,
 {\em Stark-Heegner points attached to Cartan non-split curves}, 
 McGill Ph.D thesis 2015.

 


\bibitem[Ve16]{Ve} 
Rodolfo Venerucci, {\em Exceptional zero formulae and a conjecture of Perrin-Riou},
 Inv. Math. {\bf 203} (2016), 923--972.

 
\bibitem[Wi88]{W} 
Andrew Wiles, {\em On ordinary $\Lambda$-adic representations assoc. to modular forms},
 Inv. Math. {\bf 94} (1988), 529-573.

\end{thebibliography}
\end{document}